\documentclass[11pt,a4paper,reqno]{article}

\usepackage{amssymb}
\usepackage{latexsym}
\usepackage{amsmath}
\usepackage{graphicx}
\usepackage{amsthm}
\usepackage{empheq}
\usepackage{bm}
\usepackage{booktabs}
\usepackage[dvipsnames]{xcolor}
\usepackage{pagecolor}
\usepackage{subcaption}
\usepackage{tikz,pgfplots,lipsum,lmodern}
\usepackage{enumitem}
\usepackage[most]{tcolorbox}

\usepackage[margin=2.9cm]{geometry}

\usepackage{pgf,tikz}
\usepackage{mathrsfs}
\usetikzlibrary{arrows}

\usepackage{cite}

\numberwithin{equation}{section}

\theoremstyle{definition}
\newtheorem{theorem}{Theorem}[section]
\newtheorem{corollary}[theorem]{Corollary}
\newtheorem{proposition}[theorem]{Proposition}
\newtheorem{definition}[theorem]{Definition}
\newtheorem{example}[theorem]{Example}
\newtheorem{notation}[theorem]{Notation}
\newtheorem{remark}[theorem]{Remark}
\newtheorem{lemma}[theorem]{Lemma}

\newtheorem{question}{Question}

\usepackage{calrsfs}

 \newcommand\qbin[3]{\textnormal{bin}_{#3}(#1,#2)}

\newcommand{\numberset}{\mathbb}
\newcommand{\N}{\numberset{N}}
\newcommand{\Z}{\numberset{Z}}

\newcommand{\C}{\mathcal{C}}

\newcommand{\F}{\numberset{F}}
\newcommand{\A}{\numberset{A}}

\newcommand{\fq}{\F_q}

\newcommand{\fqm}{\F_{q^m}}

\newcommand{\mA}{\mathcal{A}}
\newcommand{\mL}{\mathcal{L}}

\newcommand{\mC}{\mathcal{C}}
\newcommand{\mS}{\mathcal{S}}
\newcommand{\mG}{\mathcal{G}}
\newcommand{\mD}{\mathcal{D}}
\newcommand{\mF}{\mathcal{F}}
\newcommand{\mW}{\mathcal{W}}

\newcommand{\mB}{\mathcal{B}}
\newcommand{\mE}{\mathcal{E}}

\newcommand{\rk}{\textnormal{rk}}

\def\rank{\mathrm{rank}}

\renewcommand{\longrightarrow}{\to}

\newcommand{\drk}{d^{\textnormal{rk}}}

\newcommand*{\myproofname}{Proof of the claim}

\usepackage{enumitem}
\setitemize{itemsep=0pt}
\setenumerate{itemsep=0pt}

\usepackage{hyperref}
\usetikzlibrary[patterns]

\title{\textbf{Generalised Evasive Subspaces}}
\usepackage{authblk}

\author[1]{Anina Gruica}
\affil[1]{Eindhoven University of Technology, the Netherlands\thanks{\{a.gruica,a.ravagnani\}@tue.nl}}

\author[1]{Alberto Ravagnani}

\author[2]{John Sheekey}
\affil[2]{University College Dublin, Ireland\thanks{john.sheekey@ucd.ie}}

\author[3]{Ferdinando Zullo}
\affil[3]{Universit\`a degli Studi della Campania ``Luigi Vanvitelli'', Italy\thanks{Corresponding Author: ferdinando.zullo@unicampania.it}}

\date{}                    

\usepackage{setspace}
\begin{document}

	\definecolor{ffzzzz}{rgb}{1.,0.6,0.6}
\definecolor{zzccff}{rgb}{0.6,0.8,1.}
\definecolor{qqqqff}{rgb}{0.,0.,1.}

	\maketitle
	\thispagestyle{empty}
	
\begin{abstract}
We introduce and explore 
a new concept of evasive subspace with respect to a collection of subspaces sharing a common dimension, most notably partial spreads.
We show that this concept generalises known notions of subspace scatteredness and evasiveness.
We establish various upper bounds for the dimension of an evasive subspace with respect to arbitrary partial spreads, obtaining improvements for the Desarguesian ones. We also establish existence results for evasive spaces in a non-constructive way, using a graph theory approach. 
The upper and lower bounds we derive have a precise 
interpretation as bounds for the critical exponent of 
certain combinatorial geometries.
Finally, we investigate connections between the notion of evasive space we introduce and the theory of rank-metric codes, obtaining new results on the covering radius and on the existence of minimal vector rank-metric codes. 
\end{abstract}


\bigskip
	
\section*{Introduction}
A natural problem that arises very often in various contexts and forms within mathematics  is the following:
	

\begin{center}
	{\it Given a family $\mA$ of subsets (or subspaces) of a set (or vector space) $X$, determine the largest cardinality (or dimension) of a subset (or subspace) $U$ of $X$ whose intersection with each element of $\mA$ has cardinality (or dimension) bounded above by some number. }
\end{center}

	Classically, this problem was studied by Crapo and Rota for the case where $U$ has trivial intersection with each element of $\mA$ and was called the {\it Critical Problem}. For this reason, we call the above stated problem the {\it Generalised Critical Problem}.
	
	More recently, generalisations of the Critical Problem have arisen naturally in finite geometry and rank-metric coding theory, in the form of {\it evasive} and {\it scattered} subspaces, which are connected to a wide variety of concepts as surveyed in \cite{lavrauw2016survey,bartoli2021evasive} and discussed in Section \ref{sub:1}, and {\it $h$-scattered} subspaces, which originally arose in connection with maximum-rank-distance codes (Section \ref{sub:2}). In this paper we seek to amalgamate and extend these recent directions into one setting, and study them in a new way utilising combinatorial and geometric tools. As we will show, this generalised setting also allows us to obtain new results towards other problems, namely the {\it covering radius} of rank-metric codes, {\it cutting blocking sets} and the related \emph{minimal} rank-metric codes.

Our setup will be the following. Let $X$ be an $N$-dimensional vector space over a finite field $\fq$. Let $\mA$ be a collection of subspaces of $X$ sharing a common dimension. We wish to study the following question.

\begin{question}\label{prob:gencriprob}
Let $h$ and $k$ be two positive integers.
Does there exist a $k$-dimensional subspace $U$ of $X$ such that every element of $\mA$ meets $U$ in a subspace of dimension at most~$h$?
\end{question}

A subspace $U$ will be said to be $(\mA,h)$-\textbf{evasive} if it satisfies the conditions of Question~\ref{prob:gencriprob}. We adopt this terminology from the theory of evasive and scattered spaces with respect to spreads, a topic with many applications in finite geometry and coding theory. We wish to investigate how much (or how little) can be determined about the existence of such subspaces based only on the intersection properties of the collection~$\mA$. 
As we will explain in Subsection~\ref{subs:cexp}, Question \ref{prob:gencriprob} can be seen as a problem concerning the critical exponent of certain combinatorial geometry, the computation of which is a central problem in enumerative combinatorics.

In this paper we will focus on the case where $\mA$ is a (partial) spread; definitions will be given in the next section. Our goals include solving the following problems addressing Question~\ref{prob:gencriprob}: 
\begin{enumerate}
    \item[(P1)] Given $h$, find upper bounds on $k$ above such that the existence of $(\mA,h)$-evasive subspaces of dimension $k$ is not possible; see Section \ref{sec:upper}.
    \item[(P2)] Show the existence of (partial) spreads $\mA$ possessing an $(\mA,h)$-evasive subspace whose dimension meets these upper bounds; see Section \ref{sec:upper}.
    \item[(P3)] Find upper and lower bounds for the largest dimension of an $(\mA,h)$-evasive subspace in the case where $\mA$ is a (partial) Desarguesian spread; see Sections \ref{sec:upper} and  \ref{sec:lower}.
    \item[(P4)] Given $h$, find lower bounds on $k$ below which the existence of $(\mA,h)$-evasive subspaces of dimension $k$ is guaranteed; see Section \ref{sec:lower}.
    \item[(P5)] Determine for which parameters $(\mA,h)$-evasive subspaces of dimension $k$ are common or rare via asymptotic density results; see Section \ref{Sec:asymptotic}.
    \item[(P6)] Given $k$, find lower bounds on $h$ below which the existence of $(\mA,h)$-evasive subspaces of dimension $k$ is guaranteed; see Section \ref{sec:cov}.
    \item[(P7)] Apply the developed machinery to show the existence of new {\it cutting blocking sets}; see Section \ref{sec:cutt}.
\end{enumerate}

In the literature to date, various subcategories of these problems have been studied, in particular for $\mA$ a Desarguesian spread, under the guise of {\it scattered} and {\it evasive subspaces}. We consider this more general problem  for the following reasons: Firstly, we wish to determine to what extent the structure of the Desarguesian spread influences the behaviour of its possible intersections with other subspaces. Secondly, should a $(\mD,h)$-evasive subspace not exist, we would like to know how small the set of spread elements whose intersection is too large can be; this number can be bounded by considering partial Desarguesian spreads. Thirdly, non-Desarguesian spreads are of interest for many reasons, for example to construct  translation planes and linear spaces, and determining the existence or non-existence of $(\mA,h)$-evasive subspaces is a new and non-trivial task; we will also show that this existence question is related to calculating the covering radius of an associated rank-metric code. And finally, in applications such as the construction of short minimal codes, we end up requiring to find $(\mA,h)$-evasive subspaces for strictly partial (not necessarily Desarguesian) spreads.

Our motivations are not necessarily to construct optimal solutions to these problems, although in some cases we do have some results in this direction. Rather, our aim is to introduce and analyse the general problem, showing both the strengths and limitations of different combinatorial and geometric techniques, and to evaluate for which parameters the existence of $(\mA,h)$-evasive subspaces of dimension $k$ is possible, likely, or difficult to determine.

In Section \ref{sec:def} we will outline the necessary definitions and preliminaries. In Section \ref{sec:motivation} we motivate our study of these problems and formally introduce our generalisation which will encompass various previously studied problems.

In Section \ref{sec:upper} we will develop upper bounds and tightness results, addressing Problems~(P1-3). In Section \ref{sec:lower} we will develop lower bounds and existence results, addressing Problems~(P3-4). In Section~\ref{Sec:asymptotic} we will analyse the asymptotic density, towards Problem~(P5).

In Section~\ref{sec:cov}, a special case of Problem~(P6) will be used to obtain new non-trivial lower bounds on the {\it covering radius} of certain MRD codes. In these cases $\mathcal{A}$ has been chosen to be a partial and in general not necessarily Desarguesian spread.

In Section~\ref{sec:cutt}, we will construct new examples of linear cutting blocking sets and minimal codes by using the construction provided in Section \ref{sec:lower}, providing examples of minimal linear rank-metric codes with shortest length known for some parameters.

\paragraph*{Acknowledgements}
A. G. is supported by the Dutch Research Council through grant OCENW.KLEIN.539. A. R. is supported by the Dutch Research Council through grants VI.Vidi.203.045, 
OCENW.KLEIN.539, 
and by the Royal Academy of Arts and Sciences of the Netherlands. J. S. was partially supported by the Italian National Group for Algebraic and Geometric Structures and their Applications (GNSAGA - INdAM). F. Z. is very grateful for the hospitality of  Eindhoven University of Technology, the Netherlands, where he was a visiting researcher for two weeks during the development of this research with the support of the DIAMANT Mathematics Cluster, the Netherlands. The research of F. Z. was supported by the project ``VALERE: VAnviteLli pEr la RicErca" of the University of Campania ``Luigi Vanvitelli'' and by the project COMBINE. Also he was partially supported by the Italian National Group for Algebraic and Geometric Structures and their Applications (GNSAGA - INdAM).

\section{Definitions and Preliminaries}\label{sec:def}

Throughout this paper, $q$ denotes a prime power, $\F_q$ is the finite field with $q$ elements, $N \ge 2$ is an integer, and $X$ is an $N$-dimensional vector space over $\F_q$. Without loss of generality, we will sometimes assume $X=\F_q^N$. We denote by $h$ a non-negative integer.

\begin{definition}
For an integer $1 \le m \le N$, the set of all $m$-dimensional subspaces of $X$ is denoted by $\mG_q(m,X)$, also called the \textbf{Grassmannian}. We will write
$\mG_q(m,N)$ instead of~$\mG_q(m,\F_q^N)$.
\end{definition}

Note that if $N=mn$, then 
$\mG_{q^m}(k,n)$ can be embedded into $\mG_{q}(mk,mn)$
by regarding a $k$-dimensional $\F_{q^m}$-subspace of $\F_{q^m}^n$ as a $mk$-dimensional subspace of $\F_q^{mn}$.
We will implicitly use this embedding throughout the paper.
We will also extensively use the $q$-binomial coefficient of non-negative integers
$i \ge j$, defined as
$$\qbin{i}{j}{q}= \prod_{\ell=0}^{j-1}\frac{(q^i-q^\ell)}{\left(q^j-q^i\right)}.$$
It is well-known that $\smash{\qbin{i}{j}{q}}$ counts the number of $j$-dimensional subspaces of an $i$-dimensional space over $\F_q$, i.e., the size of 
$\mG_{q}(j,i)$.

This paper is about the following problem: Given a collection $\mA$ of subspaces of $X$ of a fixed dimension $m$, investigate the properties, and in particular the existence, of subspaces of $X$ 
that intersect
each element of $\mA$ in dimension upper bounded by some integer $h$.
We therefore propose the following concept.

\begin{definition} \label{def:ahscatt}
Let $\mA$ be a subset of $\mG_q(m,N)$, and let $U \le X$ be an $\fq$-subspace. We say that
$U$ is $(\mA,h)$-\textbf{evasive} if
\[ \dim_{\fq}(U\cap S)\leq h \quad \mbox{for all $S \in \mA$}. \]
\end{definition}

As we will illustrate in Section \ref{sec:motivation}, 
Definition~\ref{def:ahscatt}
is a natural generalisation of the notion of {\it scattered} and {\it $h$-scattered} subspaces with respect to {\it spreads}, and of evasive subspaces with respect to certain families of spaces. We recall first the definition of (partial) spreads.

\begin{definition}
A collection $\mA$ of subspaces of dimension $m$ of~$X$ is called a \textbf{partial $m$-spread} if for all $S,S' \in \mA$ we have $S \cap S'= \{0\}$. If furthermore every element of $X\setminus \{0\}$ is contained in exactly one $S \in \mA$, then $\mA$ is called an \textbf{$m$-spread}.
\end{definition}

In~\cite{segre1964teoria}, Segre proved that an $m$-spread of $X$ exists if and only if $m$ divides $N$. If that is the case and $\mathcal{A}$ is an $m$-spread of $X$, then 
\begin{align} \label{eq:spreadsize}
|\mathcal{A}|=\frac{|X\setminus\{0\}|}{q^m-1} =  \frac{q^N-1}{q^m-1}.
\end{align} 

A very special class of spreads, which exist for all admissible parameters and which are by far the most studied, are the {\it Desarguesian spreads}.
Suppose that $X$ is an $\F_{q^m}$-linear space of dimension $n$ over $\F_{q^m}$.
In~\cite{segre1964teoria}, Segre showed that 
\begin{equation} \label{calD}
    \mathcal{D}=\{\langle x \rangle_{\mathbb{F}_{q^m}} \colon x\in X\setminus\{0\}\}
\end{equation} 
is an $m$-spread in $X$. We will frequently identify the elements of $X$ with $n$-tuples $x=(x_1,\ldots,x_n)$ of elements of $\F_{q^m}$, in which case we have 
\begin{equation} \label{calD}
    \mathcal{D}=\{\{(\alpha x_1,\ldots,\alpha x_n) :\alpha\in \fqm\} \colon x\in X\setminus\{0\}\}.
\end{equation} 

Since in this paper we consider intersection properties of subspaces, the natural notion of equivalence is under the action of the general semilinear group $\mathrm{\Gamma L}(nm,q)$ on $X$.

\begin{definition}
An $m$-spread of $X$ is called \textbf{Desarguesian} if it can be obtained from \textit{the} Desarguesian spread~$\mathcal{D}$ 
in~\eqref{calD} by applying a semilinear transformation of $\mathrm{\Gamma L}(nm,q)$; see below. If a partial $m$-spread is contained in a Desarguesian spread, then it is called a \textbf{partial Desarguesian spread}.
\end{definition}

For the remainder of this paper we will denote by $\mD$ the (fixed) Desarguesian $m$-spread in $X$ defined above. Note that we are abusing notation and assuming that $m$ and $n$ are clear from context. Note moreover that since all Desarguesian $m$-spreads are equivalent under the action of $\mathrm{\Gamma L}(nm,q)$, any statement concerning intersections of subspaces with elements of a specific Desarguesian $m$-spread in $X$ will remain true for all Desargusian $m$-spreads in~$X$.

We also recall that for $N=mn$ and $n\geq 3$, an $m$-spread $\mA$ is Desarguesian if and only if it is {\it normal} (i.e., 
$\mA$ induces an $m$-spread on the subspace spanned by any two elements of $\mA$); see~\cite{lunardon1999normal}. In other words, any three spread elements span either a $2m$-dimensional space or a $3m$-dimensional space.

\begin{remark}
As mentioned previously, although Desarguesian spreads are the best known and most studied, in contrast with most of the literature on the subject we will \textit{not} restrict ourselves to Desarguesian spreads in this paper. We are generalising in two ways: partial spreads that are not necessarily complete or completable, and that are not necessarily contained in a Desarguesian spread.
\end{remark}

\begin{remark}
    In \cite{bartoli2021evasive}, the concept of \textbf{$(k,h)$-evasive} subspaces was introduced. This is the case where $\mathcal{A}$ is the set of subspaces of $X$ which are $k$-dimensional over $\fqm$. The definition in this paper is a generalisation of this concept.
\end{remark}

\section{Evasive Subspaces and Their Generalisations}\label{sec:motivation}

In this section we chart the history of the motivating problems for this work. 
We start with the notion of scattered subspace with respect to a spread (Subsection~\ref{sub:1}), also commenting
on the notion of duality in that context.
In Subsection~\ref{sub:2} we turn to $h$-scattered subspaces, while in Subsection~\ref{sub:3} we illustrate the notion of evasive subspace we propose in connection with the known ones. 
Finally, in Subsection~\ref{sub:4} we characterize evasive subspaces in a lattice theory fashion, connecting them with the  Critical Problem by Crapo and Rota.

\subsection{Scattered and Evasive Subspaces with Respect to Spreads} \label{sub:1}
In \cite{blokhuis2000scattered},
 Blokhuis and Lavrauw  introduced the notion of {\it scattered subspaces with respect to spreads}, which corresponds in our language to $(\mA,1)$-evasive subspaces, where $\mA$ is a spread.

\begin{definition}
Let $\mA$ be an $m$-spread of $X$. An $\F_q$-subspace of $X$ is called \textbf{scattered} with respect to $\mA$ if it intersects each spread element $S \in \mA$ in dimension at most one. 
\end{definition} 

Evasive subspaces with respect to spreads have found many applications, for example to translation hyperovals \cite{glynn1994laguerre}, translation caps in affine spaces \cite{bartoli2018maximum}, two-intersection sets~\cite{blokhuis2000scattered}, blocking sets \cite{ball2000linear}, translation spreads of the Cayley
generalised hexagon~\cite{marino2015translation}, finite semifields~\cite{lavrauw2011finite}, coding theory~\cite{polverino2020connections,zini2021scattered}, and graph theory~\cite{calderbank1986geometry}. In the literature it is common to refer to ``scattered subspaces'' without specifying $\mA$, when $\mA$ has been fixed or is clear from context; often $\mA$ is assumed to be a fixed Desarguesian spread, but we do not assume that in this paper.





The following upper bound was proved in \cite{blokhuis2000scattered} and shown to be tight in some cases. We will generalise this result in Theorem \ref{th:generalupper}.

\begin{theorem}[see Theorem 3.2 of \cite{blokhuis2000scattered}]\label{th:boundgenscatt}
Let $U$ be an $\fq$-subspace of $X$ which is scattered with respect to an $m$-spread $\mA$.
Then $\dim_{\fq}(U)\leq m(n-1)$. Moreover, for any $m,n$ there exist $m$-spreads of~$X$ with respect to which there exists a scattered subspace of dimension $m(n-1)$.
\end{theorem}

We have the following bound on the dimension of a scattered subspace with respect to a Desarguesian spread, which also shows that subspaces of dimension $m(n-1)$ cannot be scattered with respect to a Desarguesian spread (when $n>2$).

\begin{theorem}[see Theorem 4.3 of \cite{blokhuis2000scattered}] \label{th:boundscatt}
Let $U$ be an $\fq$-subspace of $X$ which is scattered with respect to a Desarguesian $m$-spread.
Then $\dim_{\fq}(U)\leq \lfloor {mn}/2 \rfloor $.
\end{theorem}

\begin{definition}
Scattered subspaces with respect to a Desarguesian spread $\mathcal{A}$ are called \textbf{maximum scattered subspaces} (with respect to~$\mathcal{\mA}$) when the inequality in Theorem~\ref{th:boundscatt} is tight.
\end{definition}

Maximum scattered subspaces with respect to Desarguesian spreads exist for every field size $q$, if $mn$ is even, as shown in a series of papers. 

\begin{theorem}[see~\cite{ball2000linear,bartoli2018maximum,blokhuis2000scattered,csajbok2017maximum}]\label{th:scattmn/2}
Suppose that $mn$ is even. There exists a maximum scattered subspace in~$X$ with respect to any Desarguesian $m$-spread. 
\end{theorem}

When $mn$ is odd, it is not known what in general the largest dimension that a scattered space can have and few constructions of large dimension scattered subspaces are known in the case $mn$ is odd. In particular, for some parameter sets the existence of maximum scattered subspaces is still an open problem.

\begin{example}[see Theorem 2.2.5 of \cite{lavrauw2001scattered}]\label{ex:maxscatt}
Let $n=2t$ for some integer $t$. Consider
\[ U=\left\{ \left(x_1,x_1^q,\ldots,x_t,x_t^q\right) \mid x_1,\ldots,x_t \in \mathbb{F}_{q^m} \right\}. \] 
Then $U$ is an $\mathbb{F}_q$-subspace of $X$ of dimension $\dim_{\F_q}(U)={mn}/{2}=mt$ that is scattered with respect to $\mD$, i.e., it is a maximum scattered subspace. 
\end{example}

\begin{example}[see Example 2.4 of \cite{napolitano2021linear}]
Let $n=2t+1$ for some integer $t$ and consider
\[ U=\{ (x_1,x_1^q,\ldots,x_t,x_t^q,a) \colon x_1,\ldots,x_t \in \mathbb{F}_{q^m}, \, a \in \mathbb{F}_q  \}. \]
Then $U$ is a scattered $\mathbb{F}_q$-subspace of dimension $tm+1$ with respect to the Desarguesian spread $\mD$; see~\eqref{calD}. 
\end{example}

For $n=5$ and $m=3$, the bound by Blokhuis and Lavrauw in Theorem~\ref{th:boundscatt} implies that if $U$ is a scattered $\mathbb{F}_q$-subspace, then $\dim_{\mathbb{F}_q}(U)\leq 7$. In \cite{bartoli2021evasive}, examples of scattered $\mathbb{F}_q$-subspaces of dimension $7$ in~$\mathbb{F}_{q^3}^5$ were constructed in characteristic $2$, $3$, and $5$. More recently, in \cite{lia2023short} other examples of scattered $\mathbb{F}_q$-subspaces of dimension $m+2$ in $\mathbb{F}_{q^m}^3$ were constructed, under certain assumptions on $q$ and $m$.


\subsubsection{Duality of $\fq$-Subspaces and a Characterization of Evasive Subspaces}\label{sec:duality}

Other important properties of 
evasive subspaces are related to the notion of \textit{duality}, which can be used to characterize and construct examples of such subspaces.

Let $X=\F_{q^m}^n$ and let $\sigma \colon X\times X \rightarrow \mathbb{F}_{q^m}$ be a non-degenerate reflexive sesquilinear form over $X$. Define
$\sigma' \colon X \times X \rightarrow \mathbb{F}_q$ by $\sigma':(u,v)\mapsto \mathrm{Tr}_{q^m/q}(\sigma(u,v))$.
If we regard $X$ as an $mn$-dimensional $\F_q$-vector space, then $\sigma^\prime$ turns out to be a non-degenerate reflexive sesquilinear form on $X$.
Let $\perp$ and $\perp'$ be the orthogonal complement maps defined by $\sigma$ and $\sigma'$ on the lattices of $\F_{q^m}$-linear and 
$\F_q$-linear subspaces, respectively.
The following properties hold (see \cite[Section~2]{polverino2010linear} for the details):
\begin{itemize}
    \item[(i)] $\dim_{\F_{q^m}}(W)+\dim_{\F_{q^m}}(W^\perp)=n$, for every $\F_{q^m}$-subspace $W$ of $X$.
    \item[(ii)] $\dim_{\F_{q}}(U)+\dim_{\F_{q}}(U^{\perp'})=mn$, for every $\F_{q}$-subspace $U$ of $X$.
    \item[(iii)] $W^\perp=W^{\perp'}$, for every $\F_{q^m}$-subspace $W$ of $X$.
    \item[(iv)] Let $W$ and $U$ be an $\F_{q^m}$-subspace and an $\F_q$-subspace of $X$ of dimension $s$ and $t$, respectively. Then
    \begin{equation}\label{eq:dualweight} \dim_{\F_q}(U^{\perp'}\cap W^{\perp'})-\dim_{\F_q}(U\cap W)=mn-t-sm. \end{equation}
    \item[(v)] Let $\sigma$, $\sigma_1$ be non-degenerate reflexive sesquilinear forms over $X$ and define $\smash{\sigma^\prime}$, $\smash{\sigma_1^\prime}$, $\perp$, $\perp_1$, $\perp'$, and $\smash{\perp_1'}$ as above. Then there exists an invertible $\F_{q^m}$-linear map $f$ such that $\smash{f(U^{\perp'})=U^{\perp_1'}}$, i.e. $\smash{U^{\perp'}}$ and $\smash{U^{\perp_1'}}$ are $\mathrm{GL}(X)$-equivalent. 
\end{itemize}

\begin{notation}
When $U$ is an $\F_q$-subspace of $X$, we denote by $U^{\perp}$ one of the $\F_q$-subspaces of the form~$U^{\perp'}$, where $\perp'$ is defined as the restriction to $\F_q$ of any non-degenerate reflexive sesquilinear form over $X$, as explained above.
\end{notation}

The next lemma characterizes scattered subspaces via their intersection with hyperplanes.

\begin{lemma}[see Theorem 4.2 of \cite{blokhuis2000scattered}]\label{lemma:charscatt}
Suppose that $mn$ is even.
Let $U$ be an ${mn}/2$-dimensional $\fq$-subspace of~$X$.
Then $U$ is scattered if and only if
\begin{equation}\label{eq:intermaxh=1} \dim_{\fq}(H\cap U)\in\left\{\frac{mn}{2}-m,\frac{mn}{2}-m+1\right\}
\end{equation}
for every $(n-1)$-dimensional $\F_{q^m}$-subspace $H \le X$.
\end{lemma}

As an application of Lemma~\ref{lemma:charscatt}, we obtain that the set of maximum scattered subspaces is closed under the aforementioned duality.

\begin{theorem}[see Theorem 3.5 of \cite{polverino2010linear}]
The dual of a maximum scattered subspace is a maximum scattered subspace as well.
\end{theorem}
\begin{proof}
Suppose that $U$ is a maximum scattered $\fq$-subspace. By Lemma \ref{lemma:charscatt}, if $H \le X$ is a hyperplane then 
\[ \dim_{\fq}(H\cap U)\in\left\{\frac{mn}{2}-m,\frac{mn}{2}-m+1\right\}. \]
Using now 
\eqref{eq:dualweight} we obtain
\[\dim_{\F_q}(U^{\perp}\cap H^{\perp})=\dim_{\F_q}(U\cap H)+mn-\dim_{\F_q}(U)-(n-1)m \in \{0,1\}, \]
since $H^\perp$ is a one-dimensional $\F_{q^m}$-subspace of $X$, it follows that $U^{\perp}$ is scattered and clearly $\dim_{\fq}(U^{\perp})={mn}/2$.
\end{proof}

\subsection{$h$-Scattered Subspaces}
\label{sub:2}

A generalisation of scattered spaces in $X=\F_{q^m}^n$ was given in \cite{csajbok2021generalising}, where the elements of the Desarguesian spread were replaced by $\F_{q^m}$-subspaces of higher dimension.
In this section we always work with $\smash{X=\F_{q^m}^n}$.

\begin{definition}
Let $0 < h \le n-1$ be an integer. An $\F_q$-subspace $U$ of $X$ is called $h$-\textbf{scattered} if $\langle U \rangle_{\F_{q^m}}=X$ and each $h$-dimensional $\F_{q^m}$-subspace of $X$ meets $U$ in an $\F_q$-subspace of dimension at most $h$.
When $h=1$, a $1$-scattered subspace corresponds to a~$(\mathcal{D},1)$-evasive space generating the whole space $X$. 
\end{definition}

In particular, $1$-scattered subspaces correspond to scattered subspaces with the property of spanning $\F_{q^m}^n$. These objects 
are central in the theory of scattered subspaces because of their connection
with \emph{maximum rank distance codes}; see  \cite{csajbok2017maximum,lunardon2017mrd,marino2022evasive,sheekey2016new,sheekey2020rank,zini2021scattered}.

We will need the following bound on the dimension of an $h$-scattered subspace.

\begin{theorem}[see Theorem 2.3 of \cite{csajbok2021generalising}]\label{th:boundhscattold}
Let $U$ be an $\fq$-subspace of $X$ which is $h$-scattered.
Then $\dim_{\fq}(U)\leq {mn}/({h+1})$.
\end{theorem}

As for scattered subspaces, when $h+1$ divides $mn$
we say that an $h$-scattered subspace of $X$ with dimension ${mn}/({h+1})$ is a \textbf{maximum $h$-scattered subspace}.

In \cite{csajbok2021generalising}, some constructions of maximum $h$-scattered spaces were provided. In particular, the following hold.

\begin{theorem}[see Theorems 2.6 and 3.6 of \cite{csajbok2021generalising}]\label{th:exhscatt}
\begin{itemize}
    \item[(i)] If $h+1 \textnormal{ divides } n$ and $m\geq h+1$, then there exists an $h$-scattered $\fq$-subspace in $X$ of dimension ${mn}/({h+1})$.
    \item[(ii)] If $m\geq 4$ is even and $t\geq 3$ is odd, then there exists an $(m-3)$-scattered $\fq$-subspace in $X$ of dimension $mn/(m-2)$ with $n={t(m-2)}/2$.
\end{itemize}
\end{theorem}

Note that the subspaces described in the second part of Theorem~\ref{th:exhscatt} are not included in the first part, since $m-2$ does not divide ${t(m-2)}/2$ as $t$ is odd.

We also recall the following characterization of $h$-scattered subspaces.

\begin{theorem}[see Corollary 5.2 of \cite{zini2021scattered}] \label{th:charhscatt}
Let $n,m,h$ be positive integers such that $h+1 \textnormal{ divides } mn$ and $m\geq h+3$.
Let $U$ be an ${mn}/({h+1})$-dimensional $\fq$-subspace of $X$.
Then $U$ is $h$-scattered if and only if
\begin{equation}\label{eq:intermax} \dim_{\fq}(H\cap U)\leq\frac{mn}{h+1}-m+h 
\end{equation}
for every $(n-1)$-dimensional $\F_{q^m}$-subspace $H \le X$.
\end{theorem}




\subsection{A New Generalisation of Evasive Spaces} \label{sub:3}



In this paper, we propose and study the notion of evasive subspace in Definition~\ref{def:ahscatt}.
When $h=1$ and $\mA$ is a spread, then the notion coincides with that of a scattered subspace as defined in \cite{blokhuis2000scattered}; see Subsection~\ref{sub:1}.
We can also view $h$-scattered subspaces as an instance of this generalisation; letting~$\mA_h$ denote the collection of $h$-dimensional $\F_{q^m}$-subspaces of~$X$, by definition, $U$ is $h$-scattered if and only if $U$ is $(\mA_h,h)$-evasive. On the other hand, by combining Theorem \ref{th:charhscatt} with~\eqref{eq:dualweight} we obtain the following result, which allows us to study maximum $h$-scattered subspaces via $(\mD,h)$-evasive subspaces.

\begin{theorem}[See Corollary 5.4 of \cite{zini2021scattered}] \label{thm:dualh}
Let $n,m,h$ be positive integers such that $h+1$ divides $mn$ and $m\geq h+3$.
Let $U$ be an ${mn}/({h+1})$-dimensional $\fq$-subspace of $X$.
Then $U$ is $h$-scattered if and only if $U^\perp$ is $(\mathcal{D},h)$-evasive.
\end{theorem}

The observations above allow us to study both scattered and $h$-scattered subspaces as instances of $(\mA,h)$-evasive subspaces. Studying $(\mA,h)$-evasive subspaces with respect to a partial $m$-spread $\mA$ (rather than to an arbitrary set of subspaces) appears to be both the most natural and the most ``tractable'' instance of this problem.

\subsection{The Lattice Theory View} \label{sub:4}
\label{subs:cexp}

In this subsection we briefly illustrate the connection between the concept of $(\mA,h)$-evasive subspace and lattice theory (and thus with the Critical Problem by Crapo and Rota).
We omit the poset theory background that is necessary to understand this subsection and 
refer the reader 
directly to~\cite{stanley2011enumerative}.

Throughout this subsection, we fix the set $\mA$ and the integer $h$, without remembering them in the notations.
Let 
$$\mA[h+1]=\{T \le X : \dim(T)=h+1, \, T \le S \mbox{ for some $S \in \mA$}\}.$$
Denote by $\mL$ the (poset) lattice whose elements are the subspaces of $X$ spanned by some elements of $\mA[h+1]$.
It is well-known from the theory of central subspace arrangements~\cite{bjorner1994subspace} that $\mL$ is (isomorphic to) a geometric lattice, whose atoms are the elements of $\mA[h+1]$.
The join of $V,V' \in \mL$
is the subspace $V \vee V'=V+V'$ and their meet is the subspace $V \wedge V'=\langle T \in \mA[h+1] : T \le V \cap V' \rangle$,
where the span of the emptyset is $\{0\} \le X$.
In particular, $\{0\}$ is the minimum element of $\mL$ and $\langle T : T \in \mA[h+1]\rangle$ is the maximum element of $\mL$.
We let $\mu_\mL$ be the Möbius function of $\mL$ and write $\mu_\mL(V)$ for $\mu_\mL(\{0\},V)$.

\begin{remark}
Following the notation above, a subspace $U \le X$
is $(\mA,h)$-evasive if and only if it does not contain any element of 
$\mA[h+1]$.
\end{remark}

With a non-substantial abuse of terminology, we define the 
\textbf{characteristic polynomial} of the lattice $\mL$ as
$$\chi_\mL(\lambda) = \sum_{V \in \mL} \mu_\mL(V) \, \lambda^{N-\dim(V)} \in \Z[\lambda].$$
The abuse of terminology comes from the fact that $V \mapsto \dim(V)$ is not necessarily the rank function of $\mL$, and $N$ is not necessarily the rank of $\mL$. However, it is not difficult to see that $\chi_\mL(\lambda)$ carries exactly the same information as the characteristic polynomial of $\mL$, even when these objects do not coincide. 

The following theorem is a classical result by Crapo and Rota; see e.g.~\cite{crapo1970foundations,athanasiadis1996characteristic,kung1996critical}. We include a short proof for completeness.

\begin{theorem} \label{thm:gencr}
Following the above notation, we have
\begin{equation} \label{cr}
\max\{\dim(U) : U \le X, \, U \mbox{ is $(\mA,h)$-evasive}\} = N-\min\{s \in \N : \chi_\mL(q^s) \neq 0\}.
\end{equation}
\end{theorem}
The minimum on the RHS of~\eqref{cr} is called the \textbf{critical exponent} of the lattice $\mL$, which is a classical invariant of a combinatorial geometry (viewed as a lattice). Therefore, by Theorem~\ref{thm:gencr},
this paper is essentially about estimating the critical exponent of certain geometric lattices. We conclude by offering a short proof of Theorem~\ref{thm:gencr}.

\begin{proof}[Proof of Theorem~\ref{thm:gencr}]
It is easy to see that there exists a subspace $U \le X$ of dimension at least $k$ which is $(\mA,h)$-evasive if and only if there exists a matrix $M \in \F_q^{(N-k) \times N}$ whose (right) kernel is $(\mA,h)$-evasive.

For a subspace $U \le X$, let $U^\mL \in \mL$ be the join of elements $V \in \mL$ with $V \le U$. Then, by definition, 
$U$ is $(\mA,h)$-evasive if and only if $U^\mL=\{0\}$. Consider the quantity
\begin{equation} \label{eeqq}
\Sigma= \sum_{M \in \F_q^{s \times N}} \;  \sum_{\substack{V \in \mL \\ V \le \ker(M)}} \mu_\mL(V).
\end{equation}
For every $V \in \mL$, we have that $V \le \ker(M)$ if and only if $V \le \ker(M)^\mL$.
Therefore by the properties of the M\"obius function we have that $\Sigma$ counts the number of $M \in \F_q^{s \times N}$ for which $\ker(M)$ is $(\mA,h)$-evasive. On the other hand, by exchanging the summation order in~\eqref{eeqq}
we find that 
$$\Sigma= \sum_{V \in \mL} \mu_\mL(V)\,  q^{s(N-\dim(V))} = \chi_\mL(q^s).$$
Therefore $\chi_\mL(q^s) \neq 0$ if and only if there exists $M \in \F_q^{s \times N}$ such that $\ker(M)$ is $(\mA,h)$-evasive. The latter is equivalent to the existence 
of a subspace 
$U \le X$ of dimension at least $N-s$ which is $(\mA,h)$-evasive by the observation at the beginning of the proof.
\end{proof}

\section{Upper Bounds and Tightness Results}\label{sec:upper}

In this section we establish 
some bounds on the parameters of $(\mathcal{A},h)$-evasive subspaces.
More precisely, for any partial $m$-spread $\mA$ in $X$ we obtain an upper bound on the dimension of an $(\mathcal{A},h)$-evasive space, and an improved bound in the case where $\mA$ is a spread.

\subsection{Upper Bounds for General Partial Spreads}

We start by focusing on arbitrary partial spreads. The main contribution of this subsection are Theorems~\ref{th:generalupper} and~\ref{prop:upperspreadistight}. The latter generalises~\cite[Theorem 3.2]{blokhuis2000scattered}. We start with the following upper bound on the dimension of an $(\mA,h)$-evasive space.

\begin{theorem}\label{th:generalupper}
Let $\mathcal{A}$ be a partial $m$-spread of $X$, $h$ be a positive integer with $1\leq h\leq m$, and $U$ be an $\fq$-subspace of $X$.
If $U$ is $(\mathcal{A},h)$-evasive, then
\begin{equation} \label{eee}
    \dim_{\fq}(U)\leq m(n-1)+h.
\end{equation} 
Furthermore, equality is possible only if $|\mathcal{A}|\leq q^{m(n-1)}$. In particular, if $\mathcal{A}$ is an $m$-spread of $X$, then
\[ \dim_{\fq}(U)\leq m(n-1)+h-1. \]
\end{theorem}
\begin{proof}
Suppose that $\dim_{\fq}(U)\geq m(n-1)+h+1$ and let $A \in \mathcal{A}$. Then
\[ \dim_{\fq}(U\cap A)\geq \dim_{\fq}(U)+\dim_{\fq}(A)-mn\geq h+1, \]
a contradiction.
Suppose now that we have equality in~\eqref{eee}, that is, $\dim_{\fq}(U)= mn-m+h$ and $U$ is $(\mathcal{A},h)$-evasive. Then every element of $\mathcal{A}$ meets $U$ in a subspace of dimension exactly $h$. Thus counting the nonzero vectors in $X\backslash U$ contained in some element of $\mathcal{A}$, we get that $|\mathcal{A}|(q^m-q^h)\leq q^{mn}-q^{m(n-1)+h}$, from which we get that $|\mathcal{A}|\leq q^{m(n-1)}$ since $h<m$. Noting that an $m$-spread in $X$ contains $(q^{mn}-1)/(q^m-1) > q^{m(n-1)}$ elements completes the proof.
\end{proof}


Note that the case where $\mathcal{A}$ a spread and $h=1$ was considered in \cite[Theorem 4.3]{bartoli2021evasive}, where a bound was provided on the cardinality rather than the dimension of $U$; the dimension bound can be obtained with a short argument. 

The following proposition demonstrates that the first bound in Theorem~\ref{th:generalupper} cannot be improved in general for the case of partial spreads of size at most $q^{m(n-1)}$.
\begin{proposition}\label{prop:upperistight}
There exist partial Desarguesian $m$-spreads $\mathcal{A}$ of size $q^{m(n-1)}$ in $X$ for which there exists an $(\mathcal{A},h)$-evasive subspace of dimension $m(n-1)+h$.
\end{proposition}

\begin{proof}
Let $\mathcal{A}'$ be a Desarguesian spread, and let $X_1,X_2$ be such that $X_2$ is an element of $\mathcal{A}'$, $X_1$ is an $m(n-1)$-dimensional $\fq$-subspace with $X=X_1\oplus X_2$ and $X_1$ is partioned by the elements of $\mathcal{A}'$.
Let~$H$ be an $h$-dimensional $\fq$-subspace of $X_2$. Then $U:= X_1 \oplus H$ is $(\mathcal{A},h)$-evasive for the partial spread $\mathcal{A}=\mathcal{A}'\backslash \{A\in \mathcal{A}':A\leq X_1\}$.
\end{proof}

\begin{remark}
The second bound in Theorem~\ref{th:generalupper} generalises \cite[Theorem 3.1]{blokhuis2000scattered}, which was proved for the case $h=1$ and $\mathcal{A}$ a spread. In that paper, it was shown that for any $m,n$ there exist an $m$-spread $\mathcal{A}$ and a subspace $U$ of dimension $m(n-1)$ such that $U$ is $(\mathcal{A},1)$-evasive, and so the upper bound on the dimension cannot be improved for general spreads. We will now generalise this result for all $h$, showing that the second bound cannot be improved in general. 
\end{remark}

We start with the following preliminary result.

\begin{lemma}\label{lem:givenU}
Suppose that there exists an $m$-spread $\mA$ of $X$ possessing an $(\mathcal{A},h)$-evasive subspace of dimension $k$. Then for any subspace $U$ of $X$ of dimension $k$, there exists an $m$-spread $\mA'$ such that $U$ is $(\mA',h)$-evasive.
\end{lemma}

\begin{proof}
The result follows from the observation that any two subspaces of $X$ of the same dimension are equivalent under the action of $\mathrm{GL}(mn,q)$. If $U'$ is an $(\mathcal{A},h)$-evasive subspace of dimension $k$, and $\phi\in \mathrm{GL}(mn,q)$ is such that $\phi(U')=U$, then $\mA'=\{\phi(S):S\in \mA\}$ is a spread with respect to which $U$ is $(\mathcal{A}',h)$-evasive.
\end{proof}

We are now ready to establish the generalisation of~\cite[Theorem 3.2]{blokhuis2000scattered}.

\begin{theorem}\label{prop:upperspreadistight}
There exists an $m$-spread $\mathcal{A}$ in $X$ for which there exists an $(\mathcal{A},h)$-evasive subspace of dimension $m(n-1)+h-1$.
\end{theorem}

\begin{proof}
We proceed by induction on $n \ge 2$.
We start 
by assuming $n=2$
and showing that there exists a $(\mathcal{A},h)$-evasive subspace of $X$ dimension $m+h-1$, for
a Desarguesian $m$-spread
$\mathcal{A}$ in $\fq^{2m}$. It is known (see e.g. \cite{blokhuis2000scattered}) that there exist $(\mathcal{A},1)$-evasive subspaces of dimension $m$; let $U_0$ be such a space. Let $S$ be an element of $\mathcal{A}$ having trivial intersection with $U_0$; such an element must exist, as $U_0$ can meet at most $({q^m-1})/({q-1})<q^m+1=|\mathcal{A}|$ elements of $\mathcal{A}$. Let $U_1$ be any $(h-1)$-dimensional subspace of $S$, and let $U=U_0\oplus U_1$. If $\dim_{\fq}(U\cap T)>h$ for some $T\in \mathcal{A}$, then
$\dim_{\fq}(U_0\cap T)>1$, a contradiction. Hence~$U$ is~$(\mathcal{A},h)$-evasive and $\dim_{\fq}(U)=m+h-1$. Furthermore we note that, by construction, $\dim_{\fq}(U \cap S)=h-1$.

We now assume $n \ge 3$ and 
proceed by induction following a path similar to \cite{blokhuis2000scattered}. We suppose the result is true for $n-1$; that is, there exists an $m$-spread $\mathcal{A}$ in $\smash{\fq^{m(n-1)}}$ possessing an $(\mathcal{A},h)$-evasive subspace of dimension $m(n-2)+h-1$. 
Let $\overline{\mD}$ be a Desarguesian $m$-spread in $X=\fq^{mn}$, and let $X=X_2\oplus X_{n-2}$, where $X_i$ has dimension $mi$, and $\overline{\mD}$ induces an $m$-spread $\mD_i$ in each $X_i$ for $i \in \{2,n-2\}$. Note that this is possible by well-known properties of Desarguesian spreads: we may view each of $X_i$ as $i$-dimensional spaces over~$\fqm$.

Let $U_2$ be a $(\mathcal{D}_2,h)$-evasive subspace of dimension $m+h-1$ in $X_2$, which exists because of the base case of induction ($n=2$). Let $U=U_2\oplus X_{n-2}$. Then $\dim_{\fq}(U)=m(n-2)+(m+h-1)=m(n-1)+h-1$. Let $S\in \mD_2$ be such that $\dim_{\fq}(U_2\cap S)=h-1$; we can assume that such a space exists by the observation at the end of the first paragraph of this proof. 

Now consider $X_{n-1}:= S\oplus X_{n-2}$, and let $\mD_{n-1}$ denote the $m$-spread induced by $\overline{\mD}$ in~$X_{n-1}$. We claim that $U$ is $(\overline{\mD}\backslash\mD_{n-1},h)$-evasive. Indeed, let $T\in \overline{\mD}\backslash \mD_{n-1}$, suppose that $\dim_{\fq}(U\cap T)>h$, and let $H=U\cap T$. Since $\overline{\mD}$ is Desarguesian, the space $T\oplus X_{n-2}$ meets $X_2$ in an element of $\overline{\mD}$, say $T'$. Then $H\oplus X_{n-2}\leq T\oplus X_{n-2}$ is a subspace of $U$ of dimension at least $m(n-2)+(h+1)$, and thus must meet $X_2$ in a space of dimension at least $h+1$. Therefore $T'$ meets $U_2$ in a space of dimension at least $h+1$. But this contradicts the fact that $U_2$ is $(\mathcal{D}_2,h)$-evasive.

Finally by the induction hypothesis and Lemma~\ref{lem:givenU} there exists an $m$-spread $\mD'_{n-1}$ of~$X_{n-1}$ with respect to which $U\cap X_{n-1}$ is $(\mD'_{n-1},h)$-evasive. Thus $\mD':=(\overline{\mD}\backslash \mD_{n-1})\cup \mD'_{n-1}$ is an $m$-spread of $X$, and $U$ is $(\mD',h)$-evasive, completing the proof.
\end{proof}

\begin{remark}
We note that some parts of the proof of Theorem~\ref{th:generalupper} are similar to \cite[Theorem 4.3]{bartoli2021evasive}, which applies to the case of a Desarguesian spread. Our results go further: we calculate an exact dimension bound rather than a bound on the cardinality; our results apply to partial spreads; and we show that this counting argument cannot be improved without assuming further properties of the spread, since there exist spreads for which this bound is tight.
\end{remark}

\begin{remark}
We observe that the bound $\dim_{\fq}(U)\leq m(n-1)+h$ of Theorem~\ref{th:generalupper} holds for any collection $\mA$ of $m$-spaces in an $mn$-space~$X$ (over $\F_q$). In particular, if we choose~$U$ to be a fixed subspace of dimension $m(n-1)+h$ over $\F_q$, and we count the number of elements in $\mA$, where $\mA$ is be the collection of all $m$-spaces meeting~$U$ in a subspace of dimension at most $h$ (see Lemma~\ref{lem:delta}), then asymptotically we can achieve $\smash{|\mA|\sim  \left(1-q^{-(h+1)}\right)q^{m^2(n-1)}}$ as $q\to + \infty$; this will follow from the asymptotic count performed later in Lemma \ref{lem:deltasyq}. This suggests that the intersection properties of the set $\mA$ play a major role in determining the largest dimension of an $(\mA,h)$-evasive space.
\end{remark}


\subsection{Upper Bounds for Partial Desarguesian Spreads}

In this subsection we focus on a Desarguesian spread $\mathcal{A}$, and establish an upper bound on the 
dimension 
of a $(\mathcal{A},h)$-evasive space. Throughout this subsection, we assume that~$X$ is an $\F_{q^m}$-linear vector space of dimension~$n$.
We start by recalling 
\cite[Corollary~4.9]{bartoli2021evasive}, which applies to Desarguesian spreads.

\begin{theorem}[See Corollary 4.9 of \cite{bartoli2021evasive}]\label{th:boundoldold}
Let $\mathcal{A}$ be a Desarguesian $m$-spread of $X$, $h$ be an integer such that $1\leq h\leq m$ and $U$ be an $\fq$-subspace of $X$.
If $U$ is $(\mathcal{A},h)$-evasive, then
\[ \dim_{\fq}(U)\leq \frac{hmn}{h+1}. \]
\end{theorem}

\begin{remark}
It is interesting to observe that
the bound of Theorem~\ref{th:generalupper} (which holds for general spreads) is sharper
than Theorem \ref{th:boundoldold} (which holds only for Desarguesian spreads) if and only if $n< h+1$ and $m> ({h^2-1})/({h+1-n})$. Note that, when $h=1$ this phenomena does not occur as we usually consider the case in which $n\geq 2$. 

In Figure~\ref{pic} we illustrate the regions in which each of the two bounds is sharper for a fixed $n$. The horizontal axis represents $h$, while the $y$-axis represents $m-h$. Note that for fixed $h$ much larger than $n$, the Desarguesian bound is sharper for $m\in [h,h+n]$, with the general bound sharper for $m>n+h$. We note also that the previous study of this problem has been mainly, though not exclusively, interested in the case where $h$ is small with respect to $n$.


\end{remark}

\begin{figure}[h!]
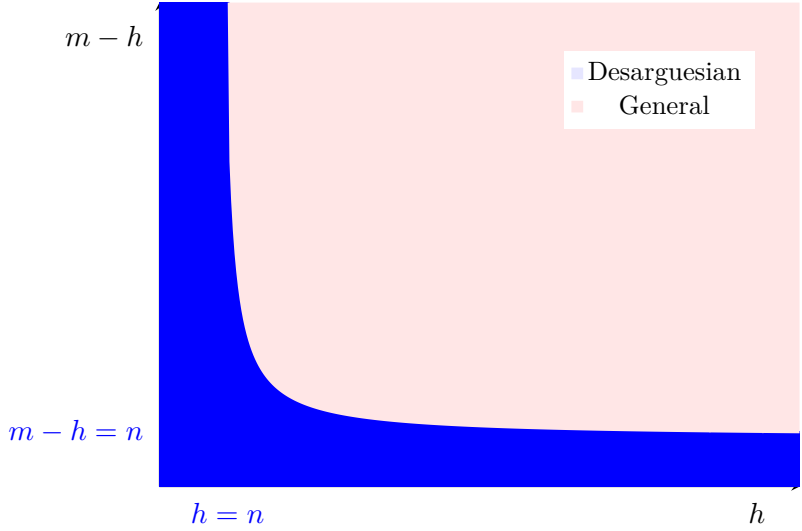

\centering

\caption{Comparison of bounds in Theorem~\ref{th:generalupper} and Theorem~\ref{th:boundoldold}}
\label{pic}
\end{figure}

We can extend the bound on the dimension of an $(\mathcal{A},h)$-evasive space
in Theorem~\ref{th:boundoldold}
to the case where $\mathcal{A}$ is a \textit{partial} Desarguesian spread by using the approach of \cite[Section~4.3]{polverino2020connections}, that is, using a suitable rank-metric code.
In the remainder of the subsection,
for $v\in X$ we denote by $\tau_{v}$ the map 
\begin{align*}
    \tau_{v}: \F_{q^m} \longrightarrow X, \qquad \lambda \mapsto \lambda v.
\end{align*}
Moreover, for a partial 
Desarguesian $m$-spread
$\mathcal{A}=\{\langle v_1\rangle_{\F_{q^m}},\ldots,\langle v_s\rangle_{\F_{q^m}}\}$ of $X$ 
we denote by $\smash{\mathcal{A}^{(2)}}$ the partial Desarguesian spread obtained from $\mathcal{A}$ as the set of all 1-dimensional $\F_{q^m}$-subspaces of $X$ which are contained in a space of the form $\smash{\langle v_i,v_j \rangle_{\mathbb{F}_{q^m}}}$ for some $i,j \in \{1,\ldots,s\}$.

Before stating the next result, we briefly recall the notion of a rank-metric code, which we will also need later in Sections \ref{sec:cov} and \ref{sec:cutt}.

\begin{definition}\label{def:rkmetric}
A \textbf{rank-metric} \textbf{code} is a subset $\mC \subseteq  \mathbb{F}_q^{m\times m'}$ with $|\mC|\geq 2$. Its \textbf{minimum distance}
is $$\drk(\mC)=\min\{\rk(A-B) \mid A,B \in \mC, \; A \neq B\}.$$ 
We say that $\mC$ is \textbf{linear} if $\mC$ is an $\F_q$-subspace of $\mathbb{F}_q^{m\times m'}$.
\end{definition}
The minimum distance and the cardinality of a rank-metric code cannot both be large at the same time. More precisely, the following holds.

\begin{theorem}[Singleton-like Bound; see \cite{delsarte1978bilinear}]
\label{thm:slb}
Let $\mC \subseteq \mathbb{F}_q^{m\times m'}$ be a rank-metric code with $\drk(\mC) \ge d$. We have
\begin{equation*} \label{singletonlikebound}
    |\mC| \le q^{\max\{m,m'\}(\min\{m,m'\}-d+1)}.
\end{equation*}
\end{theorem}

In the next result, we will identify a rank-metric code as a subset of the $\fq$-linear maps from $\F_{q}^m$ to $\F_{q}^{m'}$ (the choices of the bases for the matrix representation is irrelevant).

\begin{theorem}\label{th:codeU}
Let $\mathcal{A}=\{\langle v_1\rangle_{\F_{q^m}},\ldots,\langle v_s\rangle_{\F_{q^m}}\}$ be a partial Desarguesian $m$-spread of~$X$ of size $s$, $h$ a positive integer with $1\leq h< m$, and $U$ a $\smash{(\mathcal{A}^{(2)},h)}$-evasive subspace of~$X$ of dimension $k$.
Let $W$ be any $\fq$-vector subspace of~$X$ having dimension $mn-k$ and let $G\colon X \rightarrow W$ be any $\fq$-linear map with the property that $\ker(G)=U$.
We further let $\overline{\mathcal{A}}=\langle v_1\rangle_{\F_{q^m}}\cup\ldots\cup\langle v_s\rangle_{\F_{q^m}}$.
Define
\[ \mathcal{C}=\{ G \circ \tau_{v} : \mathbb{F}_{q^m}\rightarrow W \,|\,v\in \overline{\mathcal{A}}\}. \]
Then $|\mathcal{C}|=s(q^{m}-1)+1$ and $\drk(\mC) \geq m-h$.
Moreover, $\mathcal{C}$ is $\fq$-linear if and only if $\overline{\mathcal{A}}$ is an $\mathbb{F}_{q^m}$-subspace and $\mathcal{A}$ is an $m$-spread of $\overline{\mathcal{A}}$.
\end{theorem}
\begin{proof}
Note that for $v,w\in \overline{\mathcal{A}}$ we have that an element $\mu \in \mathbb{F}_{q^m}$ belongs to $\ker(G \circ \tau_{v}-G \circ \tau_{w})$ if and only if $G(\mu (v-w))=0$, that is, if and only if $\mu (v-w) \in U$.
Thus the fact that~$U$ is a $\smash{(\mathcal{A}^{(2)},h)}$-evasive subspace implies that $\smash{\dim_{\fq}(\ker(G \circ \tau_{v}-G \circ \tau_{w}))\leq h}$ and so $\smash{\dim_{\fq}(\mathrm{Im}(G \circ \tau_{v}-G \circ \tau_{w}))\geq m-h}$, that is, the minimum distance of $\mathcal{C}$ is at least $m-h$. 
Moreover, if $v,w \in \overline{\mathcal{A}}$ satisfy $G\circ \tau_{v}=G\circ \tau_{w}$, then $\lambda (v-w)\in U$ for every $\lambda \in \mathbb{F}_{q^m}$, hence $\langle v-w \rangle_{\F_{q^m}}\subseteq U$. The latter is a contradiction for $v \ne w$. All of this shows that $|\mathcal{C}|=|\overline{\mathcal{A}}|=s(q^{m}-1)+1$, as claimed.
We finally note that the code $\mathcal{C}$ is linear if and only if, for any $v,w \in \overline{\mathcal{A}}$, $G\circ \tau_{v}+G\circ \tau_{w}=G\circ \tau_{v+w}\in \mathcal{C}$, that is, if and only if $v+w \in \overline{\mathcal{A}}$. Hence $\mathcal{C}$ is linear if and only if $\overline{\mathcal{A}}$ is an $\mathbb{F}_{q^m}$-subspace.
\end{proof}

\begin{remark}
Following the notation of Theorem~\ref{th:codeU},
when $\mA=\mD$ is the Desarguesian \textit{spread} the code $\mC$ is $\F_q$-linear. In this situation,
\cite[Theorem 3.5]{zini2021scattered} shows how one can extract from $\mC$
a $(\mathcal{D},h)$-evasive subspace.
We notice that \cite[Theorem 3.5]{zini2021scattered} does not naturally extend to the generality of our 
Theorem~\ref{th:codeU}, as the code $\mC$ in the statement is not linear in general.
\end{remark}

As a corollary of Theorem~\ref{th:codeU}, we obtain a bound for the size of a partial Desarguesian spread $\mathcal{A}$ in the case where there exists a subspace that is $(\mathcal{A}^{(2)},h)$-evasive as well.

\begin{corollary}\label{th:boundhscatt}
Let $\mathcal{A}=\{\langle v_1\rangle_{\F_{q^m}},\ldots,\langle v_s\rangle_{\F_{q^m}}\}$ be a partial Desarguesian $m$-spread of $X$ of size $s$, $h$ be a positive integer such that $1\leq h\leq m$ and $U$ be an $\fq$-subspace of $X$ of dimension $k$.
If $U$ is $(\mathcal{A}^{(2)},h)$-evasive, then
\[ s(q^m-1)+1\leq \left\{ 
\begin{array}{ll}
    q^{(mn-k)(h+1)} & \text{if\,\,} k\leq mn-m, \\
    q^{m(mn-k-m+h+1)} & \text{otherwise.}
\end{array}
\right. \]
\end{corollary}
\begin{proof}
Let $\mathcal{C}$ be as in Theorem \ref{th:codeU}, i.e.,
\[ \mathcal{C}=\{ G \circ \tau_{v}\colon \mathbb{F}_{q^m}\rightarrow W \,|\,v\in \overline{\mathcal{A}}\}. \]
Then $\mathcal{C}$ may be also seen as a rank-metric code in $\F_q^{(mn-k)\times m}$ of size $s(q^{m}-1)+1$ and with minimum distance $d\geq m-h$. By applying the rank-metric 
Singleton bound (Theorem~\ref{thm:slb}) we obtain
\[ s(q^m-1)+1 \leq q^{\max\{mn-k,m\}(\min\{mn-k,m\}-m+h+1)},\]
which proves the desired statement.
\end{proof}

Note that when 
$\mathcal{A}$ is a Desarguesian spread, 
Corollary~\ref{th:boundhscatt} recovers
Theorem~\ref{th:exhscatt}.

\begin{remark}
The bound of 
Corollary~\ref{th:boundhscatt}
cannot be tight in general, since $s(q^m-1)+1$ is not necessarily a power of $q$. However, when $\mathcal{A}$ is a Desarguesian spread,  Corollary~\ref{cor:ex(D,h)-scatt} implies that the bound of Theorem~\ref{th:boundhscatt} is tight. 
\end{remark}

\section{Existence Results}\label{sec:lower}

This section is devoted to existence results for $(\mA,h)$-evasive subspaces.
In Subsection~\ref{sub:ex1} we focus on the case where $\mA$ is a partial Desarguesian spread, while Subsections~\ref{sec:prelim} and~\ref{sec:genexist} are devoted to arbitrary partial spreads. In the latter two subsections we derive our existence results by building on the graph theory machinery developed in~\cite{gruica2020common}.

\subsection{Existence Results for (Partial) Desarguesian Spreads} \label{sub:ex1}

We start with the following result on the existence of $(\mathcal{A},h)$-evasive subspaces of maximum dimension when $\mathcal{A}$ is a Desarguesian $m$-spread, which is a direct consequence of Theorem~\ref{th:exhscatt} and Theorem~\ref{thm:dualh}, via the duality operation described in Subsection \ref{sec:duality}.

\begin{corollary}\label{cor:ex(D,h)-scatt}
Let $\mathcal{A}$ be a Desarguesian $m$-spread in $X$.
If one of the following is satisfied:
\begin{itemize}
    \item[(i)] $h+1 \textnormal{ divides } n$ and $m\geq h+3$;
    \item[(ii)] $m\geq 4$ is even, $n=t(m-2)/2$ with $t$ an odd positive integer;
\end{itemize} 
then there exists a $(\mathcal{A},h)$-evasive $\fq$-subspace of $X$ of dimension ${hmn}/({h+1})$.
\end{corollary}

Our aim now is to provide constructions of $(\mathcal{A},h)$-evasive subspaces also in the case when $h+1$ does not divide $n$. 
\begin{corollary}\label{cor:smallhscatt}
Let $n,m,h$ be positive integers such that $m\geq h+3$, set $t=\left\lfloor {n}/({h+1})\right\rfloor$ and let $\mathcal{A}$ be a Desarguesian $m$-spread in $X$. Then there exists an $(\mathcal{A},h)$-evasive $\fq$-subspace in $X$ of dimension $hmt$.
\end{corollary}
\begin{proof}
Consider $W$ to be any $\F_{q^m}$-subspace of $X$ having dimension $t(h+1)$.
By Corollary \ref{cor:ex(D,h)-scatt} we have the existence of a $(\mathcal{A},h)$-evasive $\fq$-subspace in $X$ in $W$ of dimension ${hmt(h+1)}/({h+1})$. This space clearly is also $(\mathcal{A},h)$-evasive in $X$ and thus the statement of the corollary follows. 
\end{proof}

The following is a generalisation of~\cite[Theorem 3.1]{bartoli2018maximum}, which is a powerful tool to construct examples of evasive subspaces.

\begin{theorem}\label{th:directsum}
Let $X=X_1 \oplus \dots \oplus X_{\ell}$
where $X_i$ is a $d_i$-dimensional space over $\F_{q^m}$ with $d_i \ge 2$. If $U_i \le X_i$ is $(\mathcal{A},h)$-evasive with respect to a partial Desarguesian spread $\mathcal{A}$ in~$X$, then $U=U_1 \oplus \dots \oplus U_{\ell} \le X$ is $(\mathcal{A},h)$-evasive.
\end{theorem}
\begin{proof}
Suppose that the elements of $\mathcal{A}$ are of the form $\langle w \rangle_{\F_{q^m}}$ for some $w \in X\setminus\{0\}$.
Let $\langle w \rangle_{\F_{q^m}} =D\in \mathcal{A}$ and suppose that  $\dim_{\F_q}(D \cap U) \ge h+1$. This means that there exist $\lambda_1, \dots, \lambda_{h+1} \in \F_{q^m}$ such that $\lambda_1 w, \dots ,\lambda_{h+1} w \in D \cap U$ with $\dim_{\F_q}(\langle \lambda_1, \dots ,\lambda_{h+1} \rangle_{\fq}) =h+1$. Since $U=U_1 \oplus \dots \oplus U_{\ell}$ we have
\begin{align*}
    \lambda_i u_1 + \dots + \lambda_i u_{\ell} = \lambda_i w = u_1^i + \dots + u_{\ell}^i
\end{align*}
where $u_j, u_j^i \in U_j$ for all $i \in \{1, \dots, h+1\}$ and all $j \in \{1, \dots, {\ell}\}$. We clearly have $\lambda_i u_j = u_j^i$ for all $i \in \{1, \dots, h+1\}$ and all $j \in \{1, \dots, {\ell}\}$. Now suppose $j \in \{1, \dots, {\ell}\}$ is such that $u_j \ne 0$. Then $\lambda_1 u_j, \dots, \lambda_{h+1} u_j \in U_j$ are all non-zero and since $U_j$ is $(\mathcal{A},h)$-evasive in $X$ we have $\dim_{\F_q} (\langle \lambda_1 , \dots ,\lambda_{h+1} \rangle_{\fq}) \le h$, a contradiction.
\end{proof}

With the aid of Theorem~\ref{th:directsum} we obtain constructions of $(\mathcal{A},h)$-evasive subspaces of~$X$ (with respect to a partial Desarguesian spread $\mathcal{A}$ in $X$) which satisfy the bound of Theorem~\ref{th:boundhscatt} with equality.

\begin{proposition}
Let $X$ be the direct sum of $T_1$ and $T_2$ which are two $\F_{q^m}$-subspaces of $X$ of dimension $t_1$ and $t_2$, respectively, and let $\mathcal{A}=\{\langle v \rangle_{\F_{q^m}}\colon v \notin T_2\}$. 
Suppose that there exists a $(\mathcal{A},h)$-evasive $\fq$-subspace $U_1$ in $T_1$ of dimension ${hmt_1}/{(h+1)}$. 
Then the subspace $U=U_1\oplus T_2$ is $(\mathcal{A},h)$-evasive subspace in $X$ of dimension $({hmn+mt_2})/({h+1})$.
\end{proposition}
\begin{proof}
By Theorem \ref{th:directsum}, the subspace $U=U_1\oplus T_2$ is a $(\mathcal{A},h)$-evasive subspace in $X$ of dimension $mt_2+{hmt_1}/({h+1})$ and the assertion follows.
\end{proof}



\subsection{Graph Theory Tools and Preliminary Formulas} \label{sec:prelim}

In this section we gather some notation and tools that will be needed in Subsection~\ref{sec:genexist} and later in  Section~\ref{Sec:asymptotic}. 
We start by briefly stating some graph theory tools and refer to~\cite[Section 3]{gruica2020common} for the proofs.

\begin{definition}
A (\textbf{directed}) \textbf{bipartite graph} is a 3-tuple $\mB=(\mA,\mW,\mE)$, where $\mA$ and~$\mW$ are finite non-empty sets and $\mE \subseteq \mA \times \mW$. The elements of $\mA \cup \mW$ are the \textbf{vertices} of the graph and the $\mE$ is the set of \textbf{edges}. We call a vertex~$W \in \mW$ \textbf{isolated} if there is no $A \in \mA$ with $(A,W) \in \mE$. We say that
$\mB$ is 
\textbf{left-regular} of \textbf{degree} $\partial$ if for all $A \in \mA$
$$|\{W \in \mW \colon (A,W) \in \mE\}| = \partial.$$
\end{definition}

We want to give bounds on the number of non-isolated (and equivalently isolated) vertices in a bipartite graph. In order to do so, we need the notion of an association.

\begin{definition} \label{def:assoc}
Let $\mA$ be a finite non-empty set and let $r \ge 0$ be an integer. An \textbf{association} on $\mA$ of \textbf{magnitude} $r$ is a function 
$\alpha: \mA \times \mA \to \{0,...,r\}$ satisfying the following properties:
\begin{itemize}
\item[(i)] $\alpha(A,A)=r$ for all $A \in \mA$,
\item[(ii)] $\alpha(A,A')=\alpha(A',A)$ for all $A,A' \in \mA$.
\end{itemize}
\end{definition}

\begin{definition}
Let $\mB=(\mA,\mW,\mE)$ be a finite bipartite graph and let $\alpha$ be an association on~$\mA$ of magnitude $r$.  We say that $\mB$ is \textbf{$\alpha$-regular} if for all  $(A,A') \in \mA \times \mA$ the number of vertices $W \in \mW$ with $(A,W) \in \mE$ and 
$(A',W) \in \mE$ only depends on the evaluation of $\alpha$ for $(A,A')$. If this is the case, we denote this number by~$\mW_\ell(\alpha)$, where $\ell=\alpha(A,A')$. 
\end{definition}

Note that if $\alpha$ is an association of magnitude $r$ and we have an $\alpha$-regular graph, then this graph is necessarily left-regular of degree $\partial=\mW_r(\alpha)$.

Some important results of this paper are applications of the following two lemmas.

\begin{lemma}[\text{\cite[Lemma 3.2]{gruica2020common}}] \label{lem:upperbound}
Let $\mB=(\mA,\mW,\mE)$ be a bipartite and left-regular graph of degree $\partial \ge 0$.
Let $\mF \subseteq \mW$ be the collection of non-isolated vertices of $\mW$.
We have
$$|\mF| \le |\mA| \, \partial.$$
\end{lemma}

\begin{lemma}[\text{\cite[Lemma 3.5]{gruica2020common}}] \label{lem:lowerbound}
Let $\mB=(\mA,\mW,\mE)$ be a finite bipartite graph that is $\alpha$-regular, where $\alpha$ is an association on~$\mA$ of magnitude~$r$. Let $\mF \subseteq \mW$ be the collection of non-isolated vertices of $\mW$. If
$\mW_r(\alpha) >0$, then 
$$|\mF| \ge  \frac{\mW_r(\alpha)^2 \, |\mA|^2}{\sum_{\ell=0}^r  \mW_\ell(\alpha) \, |\alpha^{-1}(\ell)|}.$$
\end{lemma}

In the following remark we explain how the bounds on isolated vertices in bipartite graphs can be related to the problem of counting spaces that are $(\mA,h)$-evasive.

\begin{remark} \label{rem:graph}
Let $1 \le k \le N$ and $1 \le m \le N-k$ be integers and let $\mA \subseteq \mG_q(m,N)$ and $\mW = \mG_q(k,N)$. We consider the graph $\mB=(\mA,\mW,\mE)$ where~$(A,W) \in \mE$ if and only if $\dim_{\F_q}(A \cap W) \ge h+1$. Then the set of $(\mA,h)$-evasive $k$-spaces in $X$ corresponds exactly to the isolated vertices in $\mW$. 
\end{remark}

By Remark~\ref{rem:graph} we can apply Lemmas~\ref{lem:upperbound} and~\ref{lem:lowerbound} in order to give bounds on the number of $(\mA,h)$-evasive $k$-spaces, where the needed quantities to apply Lemmas~\ref{lem:upperbound} and~\ref{lem:lowerbound} are provided in the following two results (for when $\mA$ is a partial spread). Their proofs can be found in the Appendix.

\begin{lemma} \label{lem:delta}
Let $1 \le k \le N$ and $1 \le m \le N-k$ be integers and let $A$ be a space in $X$ of dimension $m$. The number of $k$-spaces in $X$ that intersect $A$ in dimension at least $h+1$ is
\begin{align*}
    \partial_q(N,k,m,h):=\sum_{\ell=h+1}^m \sum_{b=\ell}^m \qbin{m}{\ell}{q} \, \qbin{m-\ell}{b-\ell}{q} \, \qbin{N-b}{k-b}{q} \, (-1)^{b-\ell} \, q^{\binom{b-\ell}{2}}.
\end{align*}
\end{lemma}

If we consider $N,k,m,h$ and the bipartite graph $\mB=(\mA,\mW,\mE)$ to be as in Remark~\ref{rem:graph}, it is easy to see that $\partial_q(N,k,m,h)$ given in Lemma~\ref{lem:delta} is the left-degree of the graph $\mB$. Moreover, if we define $\alpha:\mA \times \mA \longrightarrow \{0,1\}$ by $\alpha(A,A') = 1$ if $A=A'$, and $\alpha(A,A') = 0$ otherwise,
then $\mB$ is $\alpha$-regular and the following Lemma provides a closed formula for $\mW_{\ell}(\alpha)$ for $\ell \in \{0,1\}$ (see Definition~\ref{def:assoc}).

\begin{lemma} \label{lem:omega}
Let $1 \le k \le N$ and $1 \le m \le N-k$ be integers, let $A,A' \le X$ be $m$-dimensional subspaces in $X$ with $A \cap A' =\{0\}$ and let $1 \le \ell \le m$ and $1 \le \ell' \le m$ be integers. The number of $k$-spaces in $X$ that intersect $A$ in dimension~$\ell$ and $A'$ in dimension~$\ell'$ is
\begin{multline*}
     \qbin{m}{\ell}{q}\qbin{m}{\ell'}{q} \sum_{r=\ell}^m \sum_{s=\ell'}^m \qbin{m-\ell}{r-\ell}{q}  \qbin{m-\ell'}{s-\ell'}{q}  \\
      \qbin{N-r-s}{k-r-s}{q}  (-1)^{r+s-\ell-\ell'} q^{\binom{r-\ell}{2}+\binom{s-\ell'}{2}}.
\end{multline*}
In particular, for an integer $1 \le h \le m-1$ the number of $k$-spaces in $X$ that intersect both~$A$ and~$A'$ in dimension at least $h+1$ is 
\begin{multline*}
     \sum_{\ell=h+1}^m\sum_{\ell'=h+1}^m \qbin{m}{\ell}{q}\qbin{m}{\ell'}{q} \sum_{r=\ell}^m \sum_{s=\ell'}^m \qbin{m-\ell}{r-\ell}{q} \qbin{m-\ell'}{s-\ell'}{q} \\
     \qbin{N-r-s}{k-r-s}{q} (-1)^{r+s-\ell-\ell'} q^{\binom{r-\ell}{2}+\binom{s-\ell'}{2}}.
\end{multline*}
and we denote this number by $\omega_q(N,k,m,h)$.
\end{lemma}

\begin{remark}
It is not hard to see that for $1 \le k \le N$ and $1 \le m \le N-k$, we have $\omega_q(N,k,m,h)=0$ whenever $h \ge \left\lfloor ({k-2})/{2}\right\rfloor$. Indeed, $h \ge \left\lfloor ({k-2})/{2}\right\rfloor$ implies that $2(h+1) \ge k$ and there cannot exist a $k$-dimensional space in $X$ simultaneously intersecting two spaces $A,A' \le X$ with $A \cap A' = \{0\}$ in dimension $h+1$ or more.
\end{remark}






\subsection{Existence Results for General Partial Spreads}\label{sec:genexist}

Using some of the results stated in Subsection~\ref{sec:prelim} we are now ready to provide bounds on the number of $(\mA,h)$-evasive spaces for a partial $m$-spread $\mA$ in $X$, and in particular show the existence of $(\mA,h)$-evasive spaces for certain dimensions. We follow the notation of Subsection~\ref{sec:prelim} and work with the graph introduced in
Remark~\ref{rem:graph}

We start with the following corollary, which is an immediate consequence of Lemma~\ref{lem:upperbound} and Lemma~\ref{lem:delta}.

\begin{corollary} \label{cor:lbShscatt}
Let $\mA$ be a partial $m$-spread in $X$ and let $1 \le h \le m-1$ be an integer. The number of $(\mA,h)$-evasive $k$-spaces in $X$ is at least
\[ \qbin{N}{k}{q}- |\mA|\,\partial_q(N,k,m,h). \]
\end{corollary}
Note that the bound of Corollary~\ref{cor:lbShscatt} heavily depends on the size of the partial spread~$\mA$. 



\begin{proposition}\label{prop:exShscatt}
Let $\mA$ be an $m$-spread in $X$ and let $1 \le h \le m-1$ be an integer. If $q \ge 64^{1/(h+1)}$, then for all $$k \le \frac{hN}{h+1}+\frac{m}{h+1}-m+h$$ there exist $(\mA,h)$-evasive $k$-spaces in~$X$.
\end{proposition}

\begin{proof}
Let $A \le X$ be any $m$-dimensional space. Then $\partial_q(N,k,m,h)$ counts the number of $k$-dimensional spaces in $X$ intersecting an $A$ in dimension $h+1$ or more. We have
\begin{align} \label{eq:estim}
    \partial_q(N,k,m,h) \le \qbin{m}{h+1}{q}\qbin{N-h-1}{k-h-1}{q}.
\end{align}
Indeed, the RHS of~\eqref{eq:estim} counts the number of $k$-dimensional spaces in $X$ that contain an $(h+1)$-dimensional subspace $S \le A$. 
Therefore, by Corollary~\ref{cor:lbShscatt}, the number of $(\mA,h)$-evasive $k$-spaces in $X$ is at least
\begin{multline*}
 \qbin{N}{k}{q}- \frac{q^N-1}{q^m-1}\,\qbin{m}{h+1}{q}\qbin{N-h-1}{k-h-1}{q} \\ =  \qbin{N}{k}{q}- \qbin{N/ m}{1}{q^m}\qbin{m}{h+1}{q}\qbin{N-h-1}{k-h-1}{q}.
\end{multline*} 
In the sequel, we will need the following estimates for the $q$-binomial coefficient:
\begin{align*}
    q^{b(a-b)} < \qbin{a}{b}{q} < 4q^{b(a-b)}, 
\end{align*}
for integers $a \ge b \ge 0$. We have
\begin{align*}
    \qbin{N}{k}{q}- \qbin{N/ m}{1}{q^m}&\qbin{m}{h+1}{q}\qbin{N-h-1}{k-h-1}{q}  \\  &> q^{k(N-k)}-4^3 q^{N-m+(h+1)(m-h-1)+(k-h-1)(N-k)} \\
    &= q^{k(N-k)}\left(1-64q^{m(N-1)+(h+1)(m-h-1-N+k)}\right) \\
    &\ge  1-64q^{N-m+(h+1)(m-h-1-N+k)}.
\end{align*}
Now note that if $k \le \frac{hN}{h+1}+\frac{m}{h+1}-m+h$ we have
\begin{align*}
    1-64q^{m(N-1)+(h+1)(m-h-1-N+k)} \ge 1-64q^{-(h+1)}
\end{align*}
and $1-64q^{-(h+1)} \ge 0$ if and only if $q^{h+1} \ge 64$, which proves the statement.
\end{proof}

By Proposition~\ref{prop:exShscatt} we have that if $\mA$ is an $m$-spread in $X$ then for all $q \ge 8$ and all $1 \le h \le m-1$ there exists an $(\mA,h)$-evasive $k$-space in~$X$.

\begin{remark}
The only known constructions for evasive spaces are with respect to a Desarguesian spread, whereas in Proposition \ref{prop:exShscatt} we consider $\mA$ to be any $m$-spread, making it a very general result.
\end{remark}

\begin{remark}
If $N=mn$ and $h+1 \nmid n$ then Proposition \ref{prop:exShscatt} shows the existence of $(\mD,h)$-evasive subspaces of dimension greater than those constructed in Corollary~\ref{cor:smallhscatt}. Indeed, if $h+1 \nmid n$ then we can write $n=t(h+1)+r$ for an integer $1 \le r < h+1$. Corollary~\ref{cor:smallhscatt} gives the existence of $(tmh)$-dimensional $(\mA,h)$-evasive spaces, whereas Proposition \ref{prop:exShscatt} gives the existence of $(\mA,h)$-evasive $k$-spaces for any $k \le \frac{hmn}{h+1}+\frac{m}{h+1}-m+h$. We have
\begin{align*}
\frac{hmn}{h+1}+\frac{m}{h+1}-m+h - tmh &= \frac{1}{h+1}\left(hmn+m-m(h+1)+h(h+1)-(h+1)tmh \right) \\
&= \frac{1}{h+1}\left(rhm+m-m(h+1)+h(h+1) \right) \\
&=\frac{h}{h+1}\left(m(r-1)+h+1\right) >0
\end{align*}
since $r \ge 1$.
\end{remark}

Combining Lemma~\ref{lem:lowerbound} and Lemma~\ref{lem:omega} we obtain the following corollary, which gives an upper bound on the number of $(\mA,h)$-evasive spaces.

\begin{corollary} \label{cor:ubShscatt}
Let $\mA$ be a partial $m$-spread in $X$ and let $1 \le h \le m-1$. Then the number of $(\mA,h)$-evasive $k$-dimensional spaces in $X$ is at most
\[ \qbin{N}{k}{q}- \frac{|\mA|\partial_q(N,k,m,h)^2}{\partial_q(N,k,m,h) + \left(|\mA|-1\right)\omega_q(N,k,m,h)}. \]
\end{corollary}
\begin{proof}
We apply Lemma~\ref{lem:lowerbound} to give a lower bound on the number of $k$-spaces in $X$ that are not $(\mA,h)$-evasive. Let $\alpha:\mA \times \mA \longrightarrow \{0,1\}$ be defined by 
\begin{align*}
    \alpha(S,S') = \begin{cases}
    1 &\textnormal{ if $S=S'$,} \\
    0 &\textnormal{ otherwise.}
    \end{cases}
\end{align*} One easily checks that $\alpha$ is an association on $\mA$ and that the bipartite graph $\mB=(\mA,\mW,\mE)$ as described in Remark~\ref{rem:graph} is $\alpha$-regular, with
\begin{align*}
    |\alpha^{-1}(0)| = |\mA|(|\mA|-1), \quad |\alpha^{-1}(1)| = |\mA|.
\end{align*}
Furthermore, by the formulas in Lemma~\ref{lem:delta} and Lemma~\ref{lem:omega} we have
\begin{align*}
    \mW_1(\alpha) =  \partial_q(N,k,m,h), \quad \mW_0(\alpha) =  \omega_q(N,k,m,h).
\end{align*}
From this we immediately get a lower bound on the number of $k$-spaces in $X$ that are \emph{not} $(\mA,h)$-evasive. This yields the upper bound on the number of $(\mA,h)$-evasive $k$-spaces in $X$ in the statement of the corollary.
\end{proof}

\section{Asymptotic Results}\label{Sec:asymptotic}

In this short section we are interested in the asymptotic behavior of the proportion of $(\mA,h)$-evasive subspaces within the set of spaces sharing a common dimension. More formally, we consider a sequence $(\mA_q)_{q \in Q}$ of partial $m$-spreads with $|\mA_q| \ge 2$ for all $q \in Q$. We want to study how the proportion of $(\mA_q,h)$-evasive spaces depends on the asymptotics of the sequence $(|\mA_q|)_{q \in Q}$ as $q \to +\infty$.
All asymptotic expressions will be stated in the following language.

\begin{notation} \label{not:landau}
We use the Bachmann-Landau notation (``Big O'', ``Little O'', and~``$\sim$'') to describe the asymptotic growth of functions defined on an infinite set of natural numbers; see e.g.~\cite{de1981asymptotic}.  We also denote by $Q$ the set of prime powers and omit ``$q \in Q$'' when writing $q \to +\infty$. 
\end{notation}

We will repeatedly use the following well-known asymptotic estimate of $q$-binomial coefficient
\begin{equation} \label{eq:qbin}
\qbin{a}{b}{q} \sim q^{b(a-b)} \quad \mbox{as $q \to +\infty$}
\end{equation}
for all integers $a \ge b \ge 0$. Moreover, we need the following two lemmas, which we prove in the Appendix.

\begin{lemma} \label{lem:deltasyq}
Let $1 \le k \le N$, $1 \le m \le N-k$ and fix $1 \le h \le m-1$. We have $$\partial_q(N,k,m,h)\sim q^{(h+1)(m-h-1)+(k-h-1)(N-k)} \quad \textnormal{as $q \to +\infty$,}$$
where $\partial_q(N,k,m,h)$ is defined in Lemma~\ref{lem:delta}.
\end{lemma}

\begin{lemma} \label{lem:omegaasyq}
Let $1 \le k \le N$, $1 \le m \le N-k$ and fix $1 \le h \le m-1$. We have $$\omega_q(N,k,m,h)\sim q^{2(h+1)(m-h-1)+(k-2h-2)(N-k)} \quad \textnormal{as $q \to +\infty$},$$
where $\omega_q(N,k,m,h)$ is defined in Lemma~\ref{lem:omega}.
\end{lemma}

The following theorem gives an answer to the problem described in the beginning of this section.

\begin{theorem} \label{prop:densescatt}
Let $(X_q)_{q \in Q}$ be a sequence of $N$-dimensional vector spaces over $\F_q$. Let $(\mA_q)_{q \in Q}$ be a sequence of partial $m$-spreads, where $A_q \subset \mG_q(m,X_q)$ for all $q \in Q$, and let $1 \le h \le m-1$ be an integer. Furthermore, let $\gamma = (h+1)(N+h+1-k-m)$. We have
\begin{align*}
    \lim_{q \to +\infty}\frac{|\{U_q \le X_q \colon \dim(U_q)=k, \, \textnormal{ $U_q$ is $(\mA_q,h)$-evasive}\}|}{\qbin{N}{k}{q}} = \begin{cases} 1 \quad &\textnormal{if  $|\mA_q| \in o\left(q^{\gamma}\right),$} \\
    0 \quad &\textnormal{if  $q^{\gamma} \in o\left(|\mA_q|\right).$}
    \end{cases}
\end{align*}
\end{theorem}
\begin{proof}
From Corollary~\ref{cor:lbShscatt} it follows that 
\begin{align*}
    \frac{|\{U_q \le X_q \colon \dim(U_q)=k, \, \textnormal{ $U$ is $(\mA_q,h)$-evasive}\}|}{\qbin{N}{k}{q}} \ge 1-\frac{|\mA_q|\,\partial_q(N,k,m,h)}{\qbin{N}{k}{q}}.
\end{align*}
By Lemma~\ref{lem:deltasyq} and the cardinality computation in~\eqref{eq:spreadsize} we have 
\begin{align*}
    |\mA_q|\,\partial_q(N,k,m,h) \sim |\mA_q|q^{(h+1)(m-h-1)+(k-h-1)(N-k)} \quad \textnormal{ as $q \to +\infty$.}
\end{align*}
In particular, if $|\mA_q| \in o\left(q^{(h+1)(N+h+1-k-a)}\right)$ as $q \to +\infty$, then we have 
\begin{align*}
    \lim_{q \to +\infty}\frac{|\mA_q|\,\partial_q(N,k,m,h)}{\qbin{N}{k}{q}} = 0,
\end{align*}
which yields the first limit in the theorem. For the second limit, suppose that $q^{\gamma} \in o\left(|\mA_q|\right)$ as $q \to +\infty$. Then by Lemma~\ref{lem:omegaasyq} we have
\begin{multline*}
    \qbin{N}{k}{q}\left(\partial_q(N,k,m,h) + \left(|\mA_q|-1\right)\omega_q(N,k,m,h)\right) \\ \sim |\mA_q|\,q^{k(N-k)+2(h+1)(m-h-1)+(k-2h-2)(N-k)}.
\end{multline*}
Moreover, we have
\begin{align*}
    |\mA_q|\,\partial_q(N,k,m,h)^2 \sim |\mA_q|\,q^{2(h+1)(m-h-1)+2(k-h-1)(N-k)} \quad \textnormal{ as $ \to +\infty$.}
\end{align*}
Therefore 
\begin{align} \label{eq:pfsecond}
    \lim_{q \to +\infty}\frac{|\mA_q|\partial_q(N,k,m,h)^2}{\qbin{N}{k}{q}\left(\partial_q(N,k,m,h) + \left(|\mA_q|-1\right)\omega_q(N,k,m,h)\right)} = 1.
\end{align}
By~\eqref{eq:pfsecond} and the upper bound on the proportion of $(\mA_q,h)$-evasive $k$-spaces within the set of $k$-spaces in $X_q$ following from Corollary~\ref{cor:ubShscatt}, the second limit of the statement follows.
\end{proof}

\begin{remark}
If we let $(X_q)_{q \in Q}$ be a sequence of $N$-dimensional vector spaces over $\F_q$ and $(\mA_q)_{q \in Q}$ be a sequence of $m$-spreads, where $A_q \subset \mG_q(m,X_q)$ for all $q \in Q$, then Proposition~\ref{prop:densescatt} gives the following asymptotic density:
\begin{align*}
    \lim_{q \to +\infty}\frac{|\{U_q \le X_q \colon \dim(U_q)=k, \; \textnormal{$U_q$ is $(\mA_q,h)$-evasive}\}|}{\qbin{N}{k}{q}} = \begin{cases} 1 \quad &\textnormal{if  $k \le \gamma,$} \\
    0 \quad &\textnormal{if  $k \ge  \gamma+2$,}
    \end{cases}
\end{align*}
where $\gamma={hN}/{(h+1)}-{mh}/(h+1)+h$. In other words, there is a tipping-point for the dimension $k$ where $(\mA_q,h)$-evasive subspaces go from being dense (and hence easy to find) to sparse (and hence difficult to find, or perhaps non-existent).

For comparison, recall from Theorem \ref{th:boundoldold} that for a Desarguesian $m$-spread $\mD$ in an $mn$-space, the maximum dimension of a $(\mD,h)$-evasive subspace is ${hmn}/({h+1})$, while the (asymptotic) tipping-point for density occurs at dimension ${hm(n-1)}/({h+1})+h$. This indicates that any construction for a subspace of dimension larger than the tipping-point could be regarded as interesting or even surprising. Furthermore, this suggests the tipping point as a lower bound for a {\it maximal} $(\mD,h)$-evasive subspace (that is, $(\mD,h)$-evasive subspace which are not properly contained in any other $(\mD,h)$-evasive subspace); this remains a topic for future research.
\end{remark}

Figure~\ref{fig:tippingpoint} illustrates this phenomenon for the case of $m=n=5$, $h=1$. We generate at random ten thousand spaces of each dimension, and calculate the proportion of these which are $(\mD,1)$-evasive, and plot the dimension $k$ on the horizontal axis and this proportion on the vertical axis. Our results imply that this should approach a step function with threshold at $k=11$ as $q$ tends to infinity. We plot the case $q=2$, and observe that already for this small field size the nature of the function is apparent.


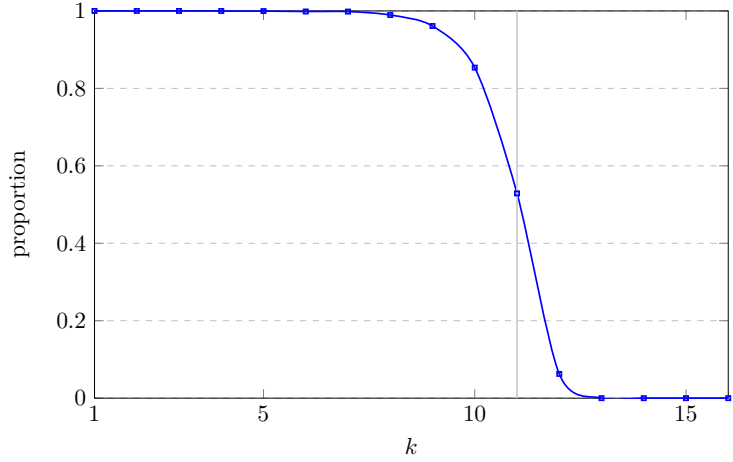
\begin{figure}[h]
\centering
\begin{tikzpicture}[scale=0.8]
\begin{axis}[legend style={at={(0.93,0.9)}, anchor = north east},
		legend cell align={left},
		width=12cm,height=8cm,
    xlabel=$k$,
    ylabel=proportion, 
    xmin=1, xmax=16,
    ymin=0, ymax=1,
    xtick={1,5,10,15},
    ytick={0,0.2,0.4,0.6,0.8,1},
    ymajorgrids=true,
    grid style=dashed,
     every axis plot/.append style={thick},  yticklabel style={/pgf/number format/fixed}
]
\addplot+[color=blue,mark=square,mark size=1pt,smooth]
coordinates {
(1,1.00000000000000000000000000000) (2,1.00000000000000000000000000000) 
(3,1.00000000000000000000000000000) (4,0.999900000000000000000000000000)
(5,0.999700000000000000000000000000) (6,0.998600000000000000000000000000)
(7,0.998400000000000000000000000000) (8,0.989500000000000000000000000000)
(9,0.961100000000000000000000000000) (10,0.853600000000000000000000000000)
(11,0.528800000000000000000000000000) (12,0.0626062606260626062606260626062)
(13,0.000000000000000000000000000000) (14,0.000000000000000000000000000000)
(15,0.000000000000000000000000000000) (16,0.000000000000000000000000000000)
    };
\addplot+[color=black!20,mark=square,mark size=0pt,smooth]
coordinates {
(11,0)
(11,1)
};
\end{axis}
\end{tikzpicture}
\caption{Proportion of $(\mD,1)$-evasive $k$-dimensional spaces in $\F_{q^5}^5$\label{fig:tippingpoint} }
\end{figure}

\section{A Connection to the Covering Radius}\label{sec:cov}


In this section we explore a new connection between $(\mA,h)$-evasive spaces 
and the \textit{covering radius} of a (possibly nonlinear) rank-metric code.
The main result of this section is Theorem~\ref{mainCR}, which states a lower bound for the covering radius of a rank-metric code.

Suppose that $N=m+m'$ with $m\leq m'$ and that $X=\F_q^N = \F_q^m \oplus \F_q^{m'}$. In the following remark we recall the well-known connection between partial $m$-spreads in $X$ and rank-metric codes in $\fq^{m\times m'}$ with minimum distance $m$; see Definition~\ref{def:rkmetric}.

\begin{remark}
Represent elements of $X$ by pairs $(x,y)\in \fq^m\oplus \fq^{m'}$. Given $A\in \fq^{m\times m'}$, we define $S_A := \{ (x,xA):x\in \fq^m \}$, which is an $m$-dimensional subspace of $X$ meeting the $m'$-dimensional space $\smash{S_\infty := 0\oplus \fq^{m'}}$ trivially. Then $\smash{\mA_\C := \{S_A:A\in \C\}}$ is a partial $m$-spread where each element meets a fixed $m'$-dimensional space trivially. Conversely, given a partial $m$-spread $\mA$ with this property, we can construct a rank-metric code $\smash{\C\subseteq \fq^{m\times m'}}$ with minimum distance $m$ such that $\mA$ is equivalent to $\mA_\C$ under the action of $\smash{\mathrm{GL}(m+m',q)}$. We specify here that the code $\C$ is not necessarily additive. The case $|\C|=q^{m'}$ corresponds to $\C$ being an MRD code; otherwise $\C$ is a {\it quasi-MRD} code; see~\cite{de2018weight}.
\end{remark}



\begin{definition}
The {\bf covering radius} of a rank-metric code $\C \subseteq \fq^{m\times m'}$ is the integer
\[
\rho(\mathcal{C}):= \max\{\min\{\rank(A-Y):A\in \mathcal{C}\}:Y\in \fq^{m\times m'}\}.
\]
\end{definition}

The covering radius for rank-metric codes has been studied e.g. in \cite{byrne2017covering, gadouleau2008packing}. For MRD codes in $\fq^{m\times m'}$ with minimum distance $m$, it was shown that the covering radius is at most $m-1$, and there exist instances with $q=2,m=4$ where the covering radius is $2<m-1$; see \cite[Example 34]{byrne2017covering}, \cite[Section 1.6]{sheekey2016new}. However, very little is known about lower bounds for the covering radius of rank-metric codes. Our techniques from studying $(\mathcal{A},h)$-evasive subspaces will allow us to obtain non-trivial lower bounds in a very general setting. We need the following result.



\begin{theorem}\label{thm:covering}
Let $\mC \subseteq \F_q^{m \times m'}$ be a rank-metric code of minimum distance $m$. The following are equivalent.
\begin{itemize}
    \item[(i)] There exists an $(\mathcal{A}_\mathcal{C},h)$-evasive $m$-space meeting $S_\infty$ trivially.
    \item[(ii)] There exists $Y\in \fq^{m\times m'}$ such that $\rank(A-Y)\geq m-h$ for all $A\in \C$.
    \item[(iii)] The covering radius of $\mathcal{C}$ is at least $m-h$.
\end{itemize}
Moreover,
\[
\rho(\mathcal{C})= m-\min\{h:\exists \, U \le X \mid \dim(U)=m, \, U\cap S_\infty=\{0\}, \, U \mbox{~is~}(\mathcal{A}_\mathcal{C},h)\mbox{-evasive}\}.
\]
\end{theorem}
\begin{proof}
Suppose $U$ is a subspace of $X$ of dimension $m$, meeting $S_\infty$ trivially. Then there exists $Y\in \fq^{m\times m'}$ such that $U=S_Y$. Furthermore, by the rank-nullity theorem we have $\dim(S_A\cap U)= m-\rank(A-Y)$. Therefore $U$ is an $(\mA_{\mC},h)$-evasive subspace if and only if $\rank(A-Y) \ge m-h$ for all $A \in \mC$. In particular, such a subspace $U$ exists if and only if there exists $Y\in \fq^{m\times m'}$ with $\rank(A-Y) \ge m-h$ for all $A \in \mC$. This shows that statements (i) and (ii) are equivalent. The equivalence of (ii) and (iii) easily follows from the definition of the covering radius.
\end{proof}

We note that $\rho(\C)=m$ if and only if $\C$ is extendable as a quasi-MRD code; that is, there exists a code $\C'$ properly containing $\C$ with minimum distance $m$. Since MRD codes are not extendable by definition, if $|\C|=q^{m'}$ then $\rho(\C)\leq m-1$. See \cite[Definition 6]{byrne2017covering} and the related definition of {\it maximality degree} for further details on this interpretation of the covering radius.

The following lemma will be needed in order to prove Proposition~\ref{prop:exShscatt2}.
\begin{lemma}[\textnormal{\cite[Section 170]{segre1961lectures}}]\label{thm:disj}
The number of $k$-dimensional spaces in $X$ disjoint from a fixed $\ell$-dimensional space $S \le X$ is $\qbin{N-k}{\ell}{q}q^{\ell k}$.
\end{lemma}

In a similar way to Proposition~\ref{prop:exShscatt}, we can prove the following result.

 \begin{proposition}\label{prop:exShscatt2}
 Let $\mA$ be a partial $m$-spread in $X$ where each element meets a fixed $m'$-dimensional space $S_\infty$ trivially, and suppose that 
 \[
 |\mA|< q^{m'(h+1)}\frac{\qbin{m}{k}{q}}{\qbin{m}{h+1}{q}\qbin{m-h-1}{k-h-1}{q}}.
 \]
 Then there exists an $(\mA,h)$-evasive $k$-subspace of $X$ meeting $S_\infty$ trivially.
 \end{proposition}

 \begin{proof}
The total number of $k$-spaces in $X$ which are disjoint from the $m'$-dimensional space~$S_\infty$ is $q^{m'k}\qbin{m}{k}{q}$; see Theorem~\ref{thm:disj}. Moreover, if we fix $A \in \mA$, then for any $(h+1)$-dimensional space $H \le A$, the number of $k$-dimensional spaces which are disjoint from $S_\infty$ and contain $H$ is $\smash{q^{m'(k-h-1)}\qbin{m-h-1}{k-h-1}{q}}$, as one can show for instance by considering the quotient space $X /H$. In particular, by applying Lemma~\ref{lem:upperbound}, the number of $k$-subspaces of $X$ that are disjoint from $S_\infty$ but are not $(\mA,h)$-evasive is at most $|\mA|\qbin{m}{h+1}{q}q^{m'(k-h-1)}\qbin{m-h-1}{k-h-1}{q}$.
  Therefore the number of $k$-subspaces of $X$ disjoint from $S_\infty$ that are $(\mA, h)$-evasive is lower bounded by
 \begin{align*}
     \qbin{m}{k}{q}q^{m'k} - |\mA|\, \qbin{m}{h+1}{q}\qbin{m-h-1}{k-h-1}{q}q^{m'(k-h-1)}.
 \end{align*}
 By the assumptions on $|\mA|$, this is positive, and so such a space exists.
 \end{proof}



\begin{corollary}\label{cor:exShscatt3}
 Let $\mA$ be a partial $m$-spread in $X$ where each element meets a fixed $m'$-dimensional space $S_\infty$ trivially, and suppose that
 \[
 |\mA| < \frac{1}{4}q^{(h+1)(m'-m+h+1)}.
 \]
Then there exists an $(\mA,h)$-evasive $m$-space in $X$ meeting $S_\infty$ trivially.
\end{corollary}

\begin{proof}
Noting as in Proposition~\ref{prop:exShscatt} that $\qbin{m}{h+1}{q} < 4q^{(h+1)(m-h-1)}$, we have that 
  \[
 |\mA|< \frac{q^{m'(h+1)}}{\qbin{m}{h+1}{q}}.
 \]
 Then applying Proposition~\ref{prop:exShscatt2} to the case $k=m$, we obtain that an $(\mA,h)$-evasive $m$-space in $X$ meeting $S_\infty$ trivially exists.
\end{proof}

\begin{theorem} \label{mainCR}
Let $\mathcal{C}\subseteq \fq^{m\times m'}$ be a rank-metric code of size $q^s$ and 
minimum distance $m$. Then the covering radius $\rho(\mathcal{C})$ is lower bounded by $m-h$, where $h$ is the smallest nonnegative integer that satisfies 
\begin{align*}
    (h+1)^2+(m'-m)(h+1)-(s+\log_q(4))>0.
\end{align*}
In particular, if $q\ne 2,3$ we obtain the following (less optimal) lower bound:
\[
\rho(\mathcal{C})\geq m-\sqrt{s+1}.
\]
\end{theorem}

\begin{proof}
Applying Corollary \ref{cor:exShscatt3} to the case $|\mA|=q^s$ gives that there exists an $(\mathcal{A}_{\mathcal{C}},h)$-evasive $m$-space meeting $S_\infty$ trivially provided $|\mA|=q^s< \frac{1}{4}q^{(h+1)(m'-m+h+1)}$. The smallest such $h$ provides a lower bound on $\rho(\mathcal{C})$ via Theorem \ref{thm:covering}. Thus we need to calculate the smallest $h$ such that 
\begin{align} \label{eq:covi}
(h+1)^2+(m'-m)(h+1)-(s+\log_q(4))>0.
\end{align}
We have
\begin{align*}
    (h+1)^2+(m'-m)(h+1)-(s+\log_q(4))&> (h+1)^2-(s+\log_q(4)) \\
    &> (h+1)^2-(s+1)
\end{align*}
where we used that for $q \ge 4$ we have $\log_q(4) \le 1$. Now $$(h+1)^2-(s+1) >0 \iff h >\sqrt{s+1}-1,$$ implying that the smallest $h$ such that~\eqref{eq:covi} holds (for any $m' \ge m$), is $h =\lceil\sqrt{s+1}\rceil$. 
\end{proof}



\begin{remark}
For example, taking $m=m'=s=6$, then we get that $\rho(\C)\geq 4$ for any MRD code $\mathcal{C}$ in $\fq^{6\times 6}$ with minimum distance $6$. The maximum possible covering radius of such a code is $5$. In \cite{sheekeyallen2022} it is shown that there exist many MRD codes in $\mathbb{F}_2^{6\times 6}$ with minimum distance $6$ having covering radius $4$. Thus this lower bound is tight in some instances. Whether this bound is tight in futher cases remains an open question.
\end{remark}

\section{Cutting Blocking Sets and Minimal Rank-Metric Codes}\label{sec:cutt}

In this section we consider an application of evasive spaces to the theory of minimal vector rank-metric codes via linear cutting blocking sets.  
One of the fundamental problems for such codes is to construct short minimal codes, since the property of being minimal is preserved by extending the code.
The main contribution of this section is a general construction of minimal 
vector rank-metric codes of parameters $\smash{[m+3,3]_{q^m/q}}$ for all $m \ge 4$.

The theory of blocking sets has been investigated from several point of view, especially in connection with applications to coding theory; see \cite{blokhuis2011blocking} for a general overview.
We recall the definition without using the projective terminologies.

\begin{definition}
Let $r,k$ be positive integers with $r<k$. An \textbf{$r$-blocking set} in $\F_q^{k}$ is a set~$\mathcal{B}$ of 1-dimensional subspaces of $\F_q^{k}$
with the property that for every $(k-r)$-subspace $\mS$ of~$\F_q^{k}$ we have 
\begin{align*}
    |\{B \in \mB \mid B \le \mS\}| \ne \emptyset.
\end{align*}
\end{definition}

In this section, we will also need the notion of a cutting $r$-blocking set,
introduced in~\cite{davydov2011linear}
with the name of \textit{strong blocking sets}.
Their connection with coding theory was explored in several recent references; see e.g.~\cite{alfarano2020geometric}.

\begin{definition}
Let $r,k$ be positive integers with $r<k$. An $r$-blocking set $\mathcal{B}$ in $\F_q^{k}$ is \textbf{$(k-r)$-cutting} if for every pair of $(k-r)$-subspaces $\mathcal{S}, \mathcal{S}'$ in $\F_q^{k}$ we have 
\[
\{B \in \mB \mid B \le \mS\} \subseteq \{B \in \mB \mid B \le \mS'\} \Longleftrightarrow \mathcal{S} = \mathcal{S}'.
\]
\end{definition}

Equivalently, the following characterization of cutting blocking sets holds.

\begin{proposition}(see Proposition 3.3 of \cite{alfarano2020geometric})
An $r$-blocking set $\mathcal{B}$ in $\F_q^{k}$ is $(k-r)$-cutting if and only if for every $(k-r)$-space $\mathcal{S}$ of $\smash{\F_q^{k}}$ we have $\langle \{B \in \mB \mid B \le \mS\} \rangle = \mathcal{S}$.
\end{proposition}


In the remainder of this section, we let $X$ denote an $\F_{q^m}$-linear vector space of dimension~$n$ over $\F_{q^m}$. In the following we recall the notion of linear sets, avoiding the projective terminologies.

\begin{definition}
A set $L$ of $1$-dimensional $\F_{q^m}$-subspaces of $X$ is said to be an \textbf{$\F_q$-linear set} of $X$ of \textbf{rank} $k$ if it is defined by the non-zero vectors of a $k$-dimensional $\F_q$-vector subspace~$U$ of~$X$, i.e., if
\[L=L_U:=\left\{\langle u \rangle_{\mathbb{F}_{q^m}} \colon u\in U\setminus \{0 \}\right\}.\]
We denote the rank of an $\fq$-linear set $L_U$ by $\mathrm{Rank}(L_U)$.
For any subspace $Y \le X$, the \textbf{weight} of $Y$ in $L_U$ is  $w_{L_U}(Y)=\dim_{\mathbb{F}_q}(U\cap Y)$.
A linear set which turns out to be also an $r$-blocking set is called a \textbf{linear $r$-blocking set}, if it is also $(k-r)$-cutting then we say that it is a \textbf{linear $(k-r)$-cutting blocking set}.
\end{definition}

We refer to~\cite{lavrauw2015field,polverino2010linear} for more details on the notions we just introduced.
In order to construct linear 2-cutting blocking sets, it is enough to construct a linear set meeting every 2-dimensional space in at least two different elements of $\mathcal{D}$.

\begin{theorem}\label{th:cutting}
Let $\mathcal{D}$ be the Desarguesian spread of $X$ defined in~\eqref{calD}.
If $U$ is a $(\mathcal{D},h)$-evasive $\fq$-subspace of $X$ having dimension $(n-2)m+h+1$, then $L_U$ is a linear 2-cutting blocking set in $X$.
\end{theorem}
\begin{proof}
Let $W$ be an $\F_{q^m}$-subspace of $X$ of dimension $2$. We have
\[ w_{L_U}(W)=\dim_{\fq}(U\cap W)=nm+h+1-\dim_{\fq}(U+W)\geq h+1. \]
Thus if $\{P \in L_U \mid P \le W\}=\{P\}$, then we would have $w_{L_U}(W)=w_{L_U}(P)\geq h+1$.
This contradicts the fact that
$U$ is a $(\mathcal{D},h)$-evasive subspace. Therefore we have $|\{P \in L_U \mid P \le W\}|\geq 2$, showing that $L_U$ is a linear cutting blocking set with respect to the 2-dimensional spaces in~$X$. 
\end{proof}

\begin{remark}
The property of being a  $(\mathcal{D},h)$-evasive is crucial in the proof of Theorem~\ref{th:cutting}, since we can easily construct a subspace $U$ of dimension $(n-2)m+h+1$ meeting some element of~$\mD$ in a space of dimension $h+1$ for which $L_U$ is not a 2-cutting blocking set in~$X$. Letting $X=X_0\oplus X_1$, where $X_0$ and $X_1$ are $\F_{q^m}$-subspace of dimension $2$ and $n-2$ respectively, and choosing $Y$ to be an $(h+1)$-dimensional subspace of an element of~$\mD$ contained in $X_0$, we have that $U=Y\oplus X_1$ does not define a 2-cutting blocking set, as the line defined by $X_0$ meets $L_U$ in only one point (cfr. Proposition \ref{prop:upperistight}).
\end{remark}

\begin{remark}
Note that it can happen that $(\mathcal{D},h)$-evasive $\fq$-subspace of $X$ having dimension $(n-2)m+h+1$ does not exist, 
e.g. when it is excluded
by the bound of Theorem~\ref{th:boundoldold}. Indeed, such subspaces can only exist when
\[(n-2)m+h+1\leq \frac{hnm}{h+1},\]
that is, when  
\[ n\leq \frac{h+1}{m}(2m-h-1). \]
\end{remark}

We have the following consequence of Theorem~\ref{th:cutting}.

\begin{corollary}\label{cor:existencecutting}
There exists linear $2$-cutting blocking sets in $X$ of rank $(n-2)m+h+1$ when $n\leq \frac{h+1}{m}(2m-h-1)$ and at least one of the following conditions holds:
\begin{enumerate}
    \item \label{cor:existencecutting1} $h+1 \textnormal{ divides } n$, $m\geq h+3$; 
    \item \label{cor:existencecutting2} $m\geq 4$ is even, $n=t(m-2)/2$ with $t$ an odd positive integer;
    \item \label{cor:existencecutting3} $h=1$, $n\leq 3$, $m\geq 4$ and $mn$ is even.
\end{enumerate}
\end{corollary}
\begin{proof}
A subspace as in Theorem \ref{th:cutting} exists under the assumptions \ref{cor:existencecutting1} and \ref{cor:existencecutting2} because of the existence of $(\mathcal{D},h)$-evasive $\fq$-subspace of $X$ having dimension $hnm/(h+1)$ by Corollary~\ref{cor:ex(D,h)-scatt}.
The sufficiency 
of assumption \ref{cor:existencecutting3} is a consequence of Theorem~\ref{th:scattmn/2}.
\end{proof}

\begin{remark}
When $h=n-1$ and 
$n\leq m$,
a $(\mathcal{D},h)$-evasive $\fq$-subspaces of $X$ of maximum dimension always exists. Indeed, one can identify $X$ with $\F_{q^m}^n$ and consider 
\[U=\left\{ \left(x,x^q,\ldots,x^{q^{n-1}}\right) \colon x \in \F_{q^m}\right\}\subseteq \mathbb{F}_{q^m}^n,\]
which is a $(n-1)$-scattered $\fq$-subspace of dimension $m$.
Its dual with respect to the standard inner product of $\F_{q^m}^n$ (cfr. Section \ref{sec:duality}) is a $(\mathcal{D},n-1)$-evasive $\fq$-subspace of dimension $(n-1)m$.
\end{remark}


As already mentioned, in this section we will construct
minimal vector rank-metric codes 
using evasive subspaces.

\begin{definition}
A \textbf{vector rank-metric code} is 
a non-zero $\F_{q^m}$-linear subspace 
$\mC \le \smash{\F_{q^m}^\ell}$.
If $\mC$ has dimension $k$, then we say that it is an $[\ell,k]_{q^m/q}$ code.
\end{definition}


\begin{notation}
For a vector $c \in \F_{q^m}^\ell$ and an ordered basis $\Gamma:=\{\beta_1, \dots, \beta_m\}$ of $\F_{q^m}$ over~$\F_q$ we let $\Gamma(c) \in \F_q^{\ell \times m}$ be defined by
\begin{align*}
    c_i = \sum_{j=1}^m \Gamma(c)_{ij} \, \beta_j \quad \mbox{for all $i \in \{1,\ldots,\ell\}$}.
\end{align*}
\end{notation}

It is easy to see that,
for all $c \in \F_{q^m}^{\ell}$, the 
column-space of~$\Gamma(c)$ does not depend on the choice of the basis $\Gamma$. This motivates the following definition.

\begin{definition}
The \textbf{rank-support} of $c \in \F_{q^m}^\ell$ is $\sigma^\rk (c) = \mbox{column-space}(\Gamma(c)) \le \F_q^\ell$, where 
$\Gamma$ is any basis of the field extension $\F_{q^m}/\F_q$.

Let $\mC \le \smash{\F_{q^m}^\ell}$ be 
a vector rank-metric code.
A codeword $c \in \mC$ is a \textbf{minimal} if $\sigma^{\mathrm{rk}}(c')\le \sigma^{\mathrm{rk}}(c)$ for some $c' \in \mathcal{C}$ implies $c'=\alpha c$ for some $\alpha \in \F_{q^m}$. We say that $\mathcal{C}$ is a \textbf{minimal vector rank-metric code} if all its codewords are minimal. 
\end{definition}

We will need the following geometric description of vector rank-metric codes via associated subspaces.

\begin{remark} \label{rem:syst}
Let $U \le \F_{q^m}^n$ be an $\fq$-subspace of dimension $\ell$ such that $\langle U \rangle_{\F_{q^m}}=\F_{q^m}^n$.
Let $G$ be a matrix in $\F_{q^m}^{n\times \ell}$ whose columns form an $\fq$-basis of $U$. 
We denote the $\F_{q^m}$-span of the rows of $G$ by $\mathcal{C}$ and equip it with the rank metric. Therefore $\mC$ is a vector 
 rank-metric code in $\F_{q^m}^\ell$ of dimension $n$.
Every element of $\mathcal{C}$ is of the form $xG$ for some $x\in \mathbb{F}_{q^m}^n$, and in \cite{randrianarisoa2020geometric} (see also \cite{alfarano2021linear}) it was shown that
\begin{equation}\label{eq:weightsystem}
\rk(xG)=\ell-\dim_{\fq}(U \cap x^\perp),
\end{equation}
where $x^\perp$ is the hyperplane of $\F_{q^m}^n$ whose equation is defined by the entries of $x$.
Conversely, 
if $\mC \le \F_{q^m}^\ell$
is a non-degenerate vector rank-metric code (i.e., where the columns of any generator matrix of $\mC$ are $\F_q$-linearly independent) and $G$ is a generator matrix of~$\mC$, then
the $\fq$-span of the columns of $G$, say $U$, satisfies~\eqref{eq:weightsystem}.
\end{remark}

Following the idea of Remark~\ref{rem:syst}, one can prove that there is a one-to-one correspondence between equivalence classes of non-degenerate $[\ell,n]_{q^m/q}$ vector rank-metric codes and equivalence classes of $\fq$-subspaces in $\F_{q^m}^n$ of dimension $\ell$ whose $\F_{q^m}$-span coincides with~$\F_{q^m}^\ell$. The latter ones are called \textbf{$[\ell,n]_{q^m/q}$ systems}. Let $\mathcal{C}$ be a non-degenerate $[\ell,n]_{q^m/q}$ vector rank-metric code with generator matrix $G$ and let $U$ be a system associated with $\mC$.
In \cite[Theorem~5.6]{alfarano2021linear} it has been proved that if $xG$ and $yG$ are codewords of~$\mathcal{C}$, then $\sigma^{\mathrm{rk}}(xG)\le \sigma^{\mathrm{rk}}(yG)$ if and only if $(x^\perp \cap U) \ge (y^\perp \cap U)$.

The connection between support inclusions and hyperplane intersections
just described
gives a one-to-one correspondence between equivalence classes of minimal non-degenerate vector rank-metric codes and linear cutting blocking sets. 
This connection was used in~\cite[Theorem~6.3]{alfarano2021linear}
in combination with constructions of 
scattered spaces of $\F_{q^m}^3$ to obtain 
non-degenerate minimal vector rank-metric codes of dimension~$3$. In the next result we recall it using the terminologies of this paper.

\begin{theorem}[see Theorem~6.3 of \cite{alfarano2021linear}]
Let $\mC$ be a non-degenerate $[\ell,3]_{q^m/q}$ code with $\ell\geq m+2$ and let $U$ be any $[\ell,3]_{q^m/q}$ system corresponding to $\mC$. If $U$ is a $(\mathcal{D},1)$-evasive $\fq$-subspace of $\F_{q^m}^{\ell}$, then $\mC$ is a minimal vector rank-metric code in $\F_{q^m}^{\ell}$. 
\end{theorem}

By combining the result 
above with Theorem \ref{th:scattmn/2} and \cite[Theorem~4.4]{blokhuis2000scattered} one obtains the following.

\begin{theorem}[see Theorem 6.7 of \cite{alfarano2021linear}]\label{th:constminimalcodes}
Suppose that $m\not\equiv 3,5 \pmod{6}$ and $m\geq 4$, then there exists a (non-degenerate) minimal $[m+2,3]_{q^m/q}$ code.
\end{theorem}

By taking 
$h=2$ and $n=3$
in our 
Corollary \ref{cor:existencecutting},
we obtain that there always exists 
a linear $2$-cutting blocking set of rank $m+3$. Therefore we have established the following result.

\begin{theorem}\label{th:newconstrctionminimal}
There exists a (non-degenerate) minimal $[m+3,3]_{q^m/q}$ code
for all $m\geq 4$. 
\end{theorem}

\begin{remark}
Although the minimal vector rank-metric codes constructed in 
\cite[Theorem~6.7]{alfarano2021linear} are shorter
than those 
of Theorem \ref{th:constminimalcodes}
(the lengths are $m+2$ and $m+3$, respectively, and the main challenge in the theory of minimal codes is to construct short codes),
in our Theorem~\ref{th:newconstrctionminimal} we have no restrictions
on $m$, while 
\cite[Theorem 6.7]{alfarano2021linear}
requires $m\not\equiv 3,5 \pmod{6}$ and $m\geq 4$. More, recently in \cite{lia2023short} other constructions have been found, but the constraints on $m$ and $q$ still remain.

Another difference between  \cite[Theorem 6.7]{alfarano2021linear} and Theorem~\ref{th:newconstrctionminimal} is that
while the construction of 
\cite{alfarano2021linear}
relies on the existence of scattered subspaces (for which often there are no explicit constructions),
for Theorem \ref{th:newconstrctionminimal} we can consider the code associated with the dual of $U=\{(x,x^q,x^{q^2})\colon x \in \F_{q^m}\}$ (cfr. Section \ref{sec:duality}). 
When $m=7$ and $q$ is odd, or $m=8$ and $q\equiv 1 \pmod{3}$,
the same goal is achieved by taking the dual of 
\[ U=\{ (x,x^q,x^{q^3}) \colon x \in \F_{q^m} \}, \]
see~\cite{csajbok2020mrd}.
Although it is very difficult to find $(m+2)$-dimensional scattered subspaces in general, based  on exhaustive searches,
such subspaces seem to  exist.
\end{remark}

\begin{remark}
In order to improve Theorem \ref{th:constminimalcodes} by constructing minimal codes of length $m+2$, it would suffice to find an $(m+2)$-space contained in the dual of $U$ which does not contain any of the $2$-spaces defined by the intersections of $U^\perp$ with the elements of~$\mD$. In other words, an $(m+2)$-space which is $(\mathcal{A},1)$-evasive with respect to a partial $2$-spread~$\mathcal{A}$ in a $2m$-dimensional space. This remains as an open problem, as well as a further illustration of the wide-ranging applicability of the generalisation introduced in this paper.
\end{remark}

\appendix

\section{Some Technical Proofs}

The main tool that is needed in this appendix is the Möbius inversion formula for the lattice of subspaces of $X$, see e.g.~\cite[Propositions 3.7.1 and Example~3.10.2]{stanley2011enumerative}. We start by recalling the statement.

\begin{lemma} \label{lem:mobius}
Let $\mG_q(N)$ be the set of all $\F_q$-subspaces of $X$, i.e. $$\mG_q(N):=\bigcup_{k=1}^N \mG_q(k,N).$$  Let $f:\mG_q(N) \longrightarrow \Z$. For all $U \in \mG_q(N)$, define $g: \mG_q(N) \longrightarrow \Z$ by $g(U) := \sum_{V \le U} f(V).$ Then we have 
\begin{align*}
    f(U) = \sum_{V \le U} g(V)\mu(V,U)
\end{align*}
for all $U \in \mG_q(N)$ where $\mu(V,U) = (-1)^{j-i}q^{\binom{j-i}{2}}$ if $\dim(V) = i$ and $\dim(U)=j$.
\end{lemma}

Note that Lemma~\ref{lem:mobius} is also true if we exchange the ``$\le$'''s with a ``$\ge$'' in the two sums and permute the variables of $\mu$. This is referred to as the dual form of the Möbius inversion formula and it will be convenient in the sequel; see e.g.~\cite[Propositions 3.7.2]{stanley2011enumerative}.

\begin{proof}[Proof of Lemma~\ref{lem:delta}]
For an $\F_q$-subspace $L \le A$ of dimension $\ell$ 
define:
\begin{align*}
    f(L) &= |\{U \le X \colon \dim(U)=k, \, U \cap A = L\}|, \\
    g(L) &= \sum_{A \ge B \ge L} f(B) = |\{U \le X \colon \dim(U)= k, \, U \cap A \supseteq L\}|.
\end{align*}
It is not hard to see that we have 
\begin{align*}
    g(L) = |\{U \le X \colon \dim(U)=k, \, U \supseteq L\}| 
    = \qbin{N-\ell}{k-\ell}{q}.
\end{align*}
By the dual form of Lemma~\ref{lem:mobius} we have
\begin{align*}
    f(L) &= \sum_{A \ge B \ge L} g(B) \mu(L,B) \\ &= \sum_{b=\ell}^m \sum_{\substack{A \ge B \ge L \\ \dim(B)=b}} \qbin{N-b}{k-b}{q} (-1)^{b-\ell}q^{\binom{b-\ell}{2}} \\
    &= \sum_{b=\ell}^m \qbin{m-\ell}{b-\ell}{q}\qbin{N-b}{k-b}{q} (-1)^{b-\ell}q^{\binom{b-\ell}{2}}.
\end{align*}
Now note that the number of $k$-spaces in $X$ that intersect $A$ in dimension $h+1$ or more is 
\begin{align*}
    \sum_{\ell=h+1}^m \sum_{\substack{L \le A \\ \dim(L)=\ell}} f(L) = \sum_{\ell=h+1}^m \sum_{b=\ell}^m \qbin{m}{\ell}{q}\qbin{m-\ell}{b-\ell}{q}\qbin{N-b}{k-b}{q}(-1)^{b-\ell}q^{\binom{b-\ell}{2}},
\end{align*}
which proves the statement.
\end{proof}

\begin{proof}[Proof of Lemma~\ref{lem:omega}]
Let $L \le A$ and $L' \le A'$ be subspaces of $A$ and $A'$, respectively with $\dim(L)=\ell$ and $\dim(L')=\ell'$. We start by counting the number of $k$-spaces $U \le X$ with $U \cap A = L$ and $U \cap A' = L'$. We define the following two maps:
\begin{align*}
    f(L,L') &= |\{U \le X \colon \dim(U)=k, \, U \cap A = L, \, U \cap A' = L'\}|, \\
    g(L,L') &= |\{U \le X \colon \dim(U)=k, \, U \cap A \supseteq L, \, U \cap A' \supseteq L'\}| \\
    &= |\{U \le X \colon \dim(U)=k, \, U \supseteq L, \, U \supseteq L'\}| \\
    &= \qbin{N-\ell-\ell'}{k-\ell-\ell'}{q}.
\end{align*}
We use the Möbius function in the product lattice of the subspaces in $A$ and $A'$ (see \cite[Proposition 3.8.2]{stanley2011enumerative}) and we get
\begin{align*}
    f(L,L') &= \sum_{\substack{A \ge \Tilde{L} \ge L \\A \ge \Tilde{L'} \ge L'}}g(\Tilde{L},\Tilde{L'})\mu(( \Tilde{L},A),(\Tilde{L'},A')) \\
    &= \sum_{r=\ell}^m \sum_{s=\ell'}^m \qbin{m-\ell}{r-\ell}{q} \qbin{m-\ell'}{s-\ell'}{q} \, \times \\ &\qquad \qquad  \qbin{N-r-s}{k-r-s}{q} (-1)^{r+s-\ell-\ell'} q^{\binom{r-\ell}{2}+\binom{s-\ell'}{2}}.
\end{align*}
Note that the number of $k$-spaces in $X$ intersecting $A$ in dimension~$\ell$ and~$A'$ in dimension~$\ell'$ is 
\begin{align*}
    \sum_{\substack{L \le A, \dim(L)=\ell \\ L' \le A', \dim(L')=\ell'}} f(L,L') = \qbin{m}{\ell}{q}\qbin{m}{\ell'}{q}f(L,L')
\end{align*}
which yields the desired result.
\end{proof}

\begin{proof}[Proof of Lemma~\ref{lem:deltasyq}]
With the aid of the asymptotic estimate in~\eqref{eq:qbin} we obtain
\begin{multline*}
     \qbin{m-\ell}{b-\ell}{q}\qbin{N-b}{k-b}{q}(-1)^{b-\ell}q^{\binom{b-\ell}{2}} \\ \sim (-1)^{b-\ell} q^{(b-\ell)(m-b)+(k-b)(N-k)+(b-\ell)(b-\ell-1)/2}
\end{multline*}
as $q \to +\infty$ for all $\ell \le b \le m$. Using elementary methods from Calculus, one shows that the map $b \mapsto (b-\ell)(m-b)+(k-b)(N-k)+(b-\ell)(b-\ell-1)/2$ attains its maximum at $b=\ell$ over the set $\{\ell,\dots,m\}$. Moreover, the value of the maximum is $(k-\ell)(N-k)$. Therefore
\begin{align*}
\sum_{b=\ell}^m \qbin{m-\ell}{b-\ell}{q}\qbin{N-b}{k-b}{q}(-1)^{b-\ell}q^{\binom{b-\ell}{2}} \sim q^{(k-\ell)(N-k)}
\end{align*}
as $q \to +\infty$. Therefore, in particular, we have
\begin{multline*}
    \qbin{m}{\ell}{q}\sum_{b=\ell}^m \qbin{m}{\ell}{q}\qbin{m-\ell}{b-\ell}{q}\qbin{N-b}{k-b}{q}(-1)^{b-\ell}q^{\binom{b-\ell}{2}} \\ \sim q^{\ell(m-\ell)+(k-\ell)(N-k)}.
\end{multline*}
as $q \to +\infty$. The map $\ell \mapsto \ell(m-\ell)+(k-\ell)(N-k)$ attains its maximum at $\ell=h+1$ over the set $\{h+1,\dots,m\}$. Putting everything together we then have
\begin{multline*}
    \sum_{\ell=h+1}^m \sum_{b=\ell}^m \qbin{m}{\ell}{q}\qbin{m-\ell}{b-\ell}{q}\qbin{N-b}{k-b}{q}(-1)^{b-\ell}q^{\binom{b-\ell}{2}} \\ \sim q^{(h+1)(m-h-1)+(k-h-1)(N-k)}
\end{multline*}
as $q \to +\infty$, which is exactly the statement of the lemma.
\end{proof}

\begin{proof}[Proof of Lemma~\ref{lem:omegaasyq}]
Fix $\ell$ and $\ell'$ such that $h+1 \le \ell, \ell' \le m$ and let $\ell \le r \le m$ and $\ell' \le s \le m$ be integers. Using the asymptotic estimate in~\eqref{eq:qbin} we have
\begin{multline*}
    \qbin{m-\ell}{r-\ell}{q} \qbin{m-\ell'}{s-\ell'}{q} \qbin{N-r-s}{k-r-s}{q} (-1)^{r+s-\ell-\ell'} q^{\binom{r-\ell}{2}+\binom{s-\ell'}{2}} \\
    \sim (-1)^{r+s-\ell-\ell'}q^{(r-\ell)(m-r)+(s-\ell')(m-s)+(k-r-s)(N-k)+\binom{r-\ell}{2}+\binom{s-\ell'}{2}}
\end{multline*}
as $q\to +\infty$.
Consider the map $$(r,s) \mapsto (r-\ell)(m-r)+(s-\ell')(m-s)+(k-r-s)(N-k)+\binom{r-\ell}{2}+\binom{s-\ell'}{2}.$$ It is not hard to see that the maximum over the integers $r \in \{\ell, \dots, m\}$ and $s \in \{\ell', \dots, m\}$ is at $r=\ell$ and $s=\ell'$ and that the value of the maximum is $(k-\ell-\ell')(N-k)$. Thus we have
\begin{multline*}
    \sum_{r=\ell}^m \sum_{s=\ell'}^m \qbin{m-\ell}{r-\ell}{q} \qbin{m-\ell'}{s-\ell'}{q}\qbin{N-r-s}{k-r-s}{q} (-1)^{r+s-\ell-\ell'} q^{\binom{r-\ell}{2}+\binom{s-\ell'}{2}}  \\
     \sim q^{(k-\ell-\ell')(N-k)}
\end{multline*}
as $q\to +\infty$. Similar arguments, which are left to the reader, yield the asymptotic behavior of $\omega_q(N,k,m,h)$ as $q \to +\infty$.
\end{proof}

\bigskip

\bibliographystyle{amsplain}
\bibliography{ourbib}

\end{document}